\documentclass[12pt,a4paper]{amsart}

\usepackage[utf8]{inputenc}
\usepackage[sc,osf]{mathpazo}
\usepackage[T1]{fontenc}
\usepackage[scaled=0.85]{beramono}
\usepackage[euler-digits]{eulervm}
\usepackage{amssymb}
\usepackage{mathtools}
\usepackage[a4paper,centering]{geometry}
\usepackage[pdfdisplaydoctitle,colorlinks,breaklinks,urlcolor=blue,linkcolor=blue,citecolor=blue]{hyperref}
\usepackage[capitalise,nameinlink]{cleveref}
\crefname{equation}{}{}
\usepackage{enumitem}
\usepackage[english]{babel}
\newtheorem{theorem}{Theorem}[section]
\newtheorem{lemma}[theorem]{Lemma}
\newtheorem{proposition}[theorem]{Proposition}

\newtheorem{conjecture*}{Conjecture}
\theoremstyle{definition}
\newtheorem{definition}[theorem]{Definition}

\theoremstyle{remark}
\newtheorem{remark}[theorem]{Remark}
\numberwithin{equation}{section}
\newcommand{\divg}{\operatorname{div}}
\def\<{\langle}
\def\>{\rangle}
\def\d{{\,\rm d}}

\newcommand{\E}{\mathbb{E}}
\newcommand{\N}{\mathbb{N}}
\newcommand{\R}{\mathbb{R}}
\newcommand{\Z}{\mathbb{Z}}
\newcommand{\Tb}{\mathbb{T}}
\newcommand{\Pb}{\mathbb{P}}

\newcommand{\Fc}{\mathcal{F}}
\newcommand{\Lc}{\mathcal{L}}

\usepackage{scalerel}
\usepackage{stackengine}
\def\avint{%
	\,\ThisStyle{%
		\ensurestackMath{%
			\stackinset{c}{0\LMpt}{c}{0\LMpt}{\SavedStyle-}{\SavedStyle\phantom{\int}}
		}%
		\setbox0=\hbox{$\SavedStyle\int\,$}\kern-\wd0
	}%
	\int
}

\usepackage{tcolorbox}

\usepackage{mathscinet}
\def\MRnum#1\empty{#1}
\renewcommand{\MRhref}[2]{%
	\href{http://www.ams.org/mathscinet-getitem?mr=#1}{#2}
}
\renewcommand{\MR}[1]{
	\relax\ifhmode\unskip\space\fi
	\MRhref{\MRnum#1\empty}{\texttt{\Tiny[MR\MRnum#1\empty]}}
}

\makeatletter
\@namedef{subjclassname@2020}{2020 Mathematics Subject Classification}
\makeatother

\begin{document}
	\title[]{Delayed Blow-up in 3D Fluids via Pseudo-Transport Noise}
	\author[S. Jiao]{Shuaijie Jiao}
	\author[M. Romito]{Marco Romito}
	\address{School of Mathematical Sciences, University of Chinese Academy of Sciences,
		Beijing 100049, China, and Academy of Mathematics and Systems Science, Chinese Academy of Sciences, Beijing
		100190, China}
	\address{Dipartimento di Matematica, Universit\`a di Pisa, Largo Bruno Pontecorvo 5, 56127 Pisa, Italia }
	\email{\href{mailto:jiaoshuaijie@amss.ac.cn}{jiaoshuaijie@amss.ac.cn}}
	\email{\href{mailto:marco.romito@unipi.it}{marco.romito@unipi.it}}
	\urladdr{\url{http://people.dm.unipi.it/romito}}
	
	\begin{abstract}
         We establish scaling limit results for fluid dynamics equations driven by pseudo-transport noise. The behaviour of noise at small scales is governed by a parameter $a$. This extends previous results of \cite{FlaLuo20, Gal20}, which corresponds to $a=0$ in our setting. Depending on the value of $a$, we prove that the noise delays the potential blow-up of both the 3D Euler and Navier-Stokes (NS) equations with high probability.
	\end{abstract}
	\maketitle
	\tableofcontents
	
	\section{Introduction}

	We explore the effect of pseudo-transport noise on the 3D incompressible fluid equations on $\Tb^3$:
	\begin{equation*}
 	    \begin{cases}
	 	    \d u + (u\cdot\nabla u)\,\d t + \nabla p\,dt = \kappa\Delta u\,\d t + \sum_k \Lc_k u\circ \d W^k,\\
            \divg u = 0,\\
	 		u(0)=u_0,
	 	\end{cases}
	  \end{equation*}
    where $u$ is the velocity field and $\kappa>0$ (resp. $\kappa=0$) corresponds to the Navier-Stokes equations (resp. the Euler equations), and noise is assumed to be divergence-free, see \eqref{eq:Lk} below. Here, $\{W^k\}$ is a sequence of independent standard Brownian motions and $\circ$ denotes the Stratonovich integral.
    The above equation will be understood as projected onto the space of divergence-free vector fields through the Leray projection $\Pi$ (defined in \eqref{Leray} below), namely
	\begin{equation}\label{fluid eq}
	 	\d u + \Pi(u\cdot\nabla u)\,\d t= \kappa\Delta u\,\d t + \sum_k \Lc_k u\circ \d W^k.\\
	\end{equation}
    To specify the structure of the noise, let $\nu>0$ be a fixed noise intensity, and $a,b \in \mathbb{R}$ be parameters. Given a sequence of divergence-free, smooth vector fields $\{\sigma_k\}$ belonging to the set
	\begin{equation}
		\mathcal{S}^\nu:=\bigg\{\{\sigma_k\}:\sigma_k\in C_{\divg}^\infty(\Tb^3;\R^3), \sum_k\|(-\Delta)^{\frac{a}{2}}\sigma_k\|_{L^\infty}^2\leq \nu\bigg\},
	\end{equation}
    the operator $\mathcal{L}_k$ is defined as
    \begin{equation}\label{eq:Lk}
      \Lc_k u
        = (-\Delta)^b\Pi\big(\sigma_k\cdot\nabla(-\Delta)^{a+b}u \big).
    \end{equation}
    Note that the specific choice $a=b=0$ recovers the transport noise. Additionally, we assume throughout this paper that both $u$ and $\{\sigma_k\}$ have zero average over $\Tb^3$. Moreover, for the Navier-Stokes case ($\kappa > 0$), we set $\kappa = 1$ for simplicity, as its specific value is irrelevant to our results.

    \subsection{Background}
    Global regularity for the deterministic 3D incompressible Navier-Stokes (NS)  and Euler equations is one of the most significant open problems in the analysis of fluid equations. The only known globally controlled quantity is given by the energy, 
    \begin{equation*}
    	\|u(t)\|^2 + 2\kappa\int_{0}^{t}\|\nabla u(s)\|_{L^2}^2\d s \leq \|u_0\|_{L^2}^2,\quad \text{for every}\quad t\geq 0.
    \end{equation*}
    However, both the $L_t^\infty L^2$ and $L_t^2 H^1$ bounds are supercritical and insufficient to exclude the possibility of blow-up. For the NS equations, the known largest space for initial data in which local well-posedness holds is the critical space $BMO^{-1}$, as shown by Koch and Tataru \cite{KT01}. In contrast, for the larger critical space $B_{\infty,\infty}^{-1}$, Bourgain and Pavlovi\'c \cite{BP08} proved the strong ill-posedness of the NS equations in the sense of norm inflation. For the Euler equations, while local well-posedness is classically known to hold in $H^s$ for $s > \frac{5}{2}$, it was shown in \cite{BL15} by Bourgain and Li that the norm inflation also occurs for the Euler equations with critical $H^{\frac{5}{2}}$ initial data. Recently, Luo \cite{LuoXiao26} extended the norm inflation results to the supercritical setting, covering $H^s$ initial velocity for the NS equations ($s \in (0, \frac{1}{2})$) as well as for the Euler equations ($s\in (0, \frac{5}{2})$).
    
    In recent years, significant progress has been made toward understanding the global regularity of fluid equations \eqref{fluid eq}, although the problem remains far from being fully resolved. For the 3D NS equations, global well-posedness is classical for small initial data in most (sub-) critical spaces. In \cite{CGP10}, Chemin et al. constructed global smooth solutions to the NS equations for a class of initial data which may be arbitrarily large in $B_{\infty,\infty}^{-1}$.  In the opposite direction, Tao \cite{Tao16} constructed a finite-time blow-up solution for an averaged version of the NS equations, which shares all the classical harmonic analysis estimates that are known to be valid for the original NS equations. Due to the absence of viscosity, there are more extensive results concerning the blow-up of the Euler equations. Notably, Elgindi \cite{Elg21} established the finite-time singularity formation for $C^{1, \alpha}$ solutions to the 3D Euler equations on $\mathbb{R}^3$. In the presence of boundaries, Chen and Hou demonstrated blow-up for $C^{1, \alpha}$ velocity fields \cite{CH21}, and more recently for $C^\infty$ data by establishing the nonlinear stability of an approximate self-similar profile \cite{CH23}.
    
    Given these challenges in the deterministic theory, there has been a growing interest in the phenomenon of regularization by noise, focusing on how different stochastic structures can restore well-posedness where deterministic PDE models fail. Regarding additive noise, Flandoli and Romito \cite{FR02, FR08} developed a Caffarelli–Kohn–Nirenberg theory and constructed Markov selections satisfying the strong Feller property for the 3D NS equations. In the case of stochastic Euler equations driven by linear multiplicative noise of large intensity, Glatt et al. \cite{GNV14} demonstrated that the solution is global with high probability. The point is that the stochastic equation can be transformed into a (random) Euler equation with damping. Recently, Bagnara et al. \cite{BMX25} proved global well-posedness for stochastic Euler equations with nonlinear It\^o noise, later extended to the Stratonovich case in \cite{Bag25}.

    In recent years, Stratonovich-type transport noise has been widely accepted as a physically relevant perturbation in fluid dynamics. Such noise can be rigorously derived from stochastic variational principles \cite{Hol15} and models small-scale fluctuations of fluid motions \cite{FP22}. In the seminal work \cite{FGP10}, it was first shown that transport noise can improve the well-posedness theory of linear transport equations with H\"older continuous drift vector field. Subsequently, Flandoli et al. demonstrated that transport noise (with suitable spatial structure) prevents the collapse of certain interacting particle systems \cite{FGP11, DFV14}. More recently, Coghi and Maurelli \cite{CM23} proved that rough Kraichnan noise restores path-wise uniqueness for the 2D Euler equations with $L^p$ vorticity. This result was further extended to the generalized SQG equations in \cite{BGM25, JL25}. 
    
    Regarding the delayed blow-up By transport noise, several key results have emerged. In \cite{FlaLuo21}, Flandoli and Luo proved that large-intensity transport noise can delay the blow-up of strong solutions of the 3D NS equation in vorticity form, thereby achieving global existence with high probability. The key to the proof lies in extending the scaling limit of the transport noise \cite{FlaLuo20,Gal20} to the nonlinear setting, where an additional diffusion term emerges in the limit. Building on this, Agresti \cite{Agr26} proved a similar delayed blow-up result for 3D NS equation in velocity form with transport noise and small hyperviscosity, by combining the scaling limit trick and stochastic maximal regularity (SMR) theory developed in \cite{VNVW12A} and a series of follow-up articles. Inspired by these developments, the present work generalizes the scaling limit of transport noise ($a=b=0$ in \eqref{fluid eq}) to the case of pseudo-transport noise, establishing delayed blow-up for both 3D NS and Euler equations. We also note that the scaling limit for transport noise was recently extended to compact Riemannian manifolds in \cite{Hua25}.
    
    \subsection{Main results}
    To be clearer, we first discuss the scaling limit of pseudo-transport noise before introducing our main results on delayed blow-up. Consider the following equation driven by pseudo-transport noise:
    \begin{equation}\label{intro-scalar eqn}
    	\d \omega =\sum_k (-\Delta)^{-b}\big( \sigma_k\cdot\nabla (-\Delta)^{a+b} \omega\big) \circ \d W^k.
    \end{equation}
    Similar to the transport noise case, at least formally, the quantity $\|\omega(t)\|_{H^{2b+a}}$ is conservative and the measure $(-\Delta)^{-\frac{a}{2}-b}\mu$ is stationary, where $\mu$ is the white noise on $\Tb^d$. Indeed, applying It\^o's formula to $\|\omega(t)\|_{H^{2b+a}}^2$, we get
    \begin{equation}\label{intro-conservation}
        \d \|\omega\|_{H^{2b+a}}^2 = 2\sum_k \big\<(-\Delta)^a \omega, \sigma_k\cdot \nabla(-\Delta)^a\omega\big\>\circ\d W^k=0,\quad \Pb\text{-a.s.}
    \end{equation}
    For the verification of stationarity for $(-\Delta)^{-\frac{a}{2}-b}\mu$, see Proposition \ref{conservation by noise}. Then by carefully choosing $\{\sigma_k^N\}\in \mathcal{S}^\nu$, we can show that the corresponding solution $\omega^N$ to \eqref{intro-scalar eqn} with the same initial data $\omega_0$ converges in law to a process $\bar{\omega}$ satisfying
    \begin{equation*}
    	\begin{cases}
    		\partial_t\bar{\omega} = -\nu (-\Delta)^{1+a}\bar{\omega}, &\text{ if } \omega_0\in H^{2b+a},\\
    		\partial_t\bar{\omega} = -\nu (-\Delta)^{1+a}\bar{\omega} + \sqrt{2\nu}(-\Delta)^{1/2-b}\dot{W}, &\text{ if } \omega_0\sim (-\Delta)^{-\frac{a}{2}-b}\mu,
    	\end{cases}
    \end{equation*}
    where $\dot{W}$ is the space-time white noise. 
    Evidently, $a$ and $b$ characterize the dissipation and fluctuation of the limit process, respectively. This extends the results of \cite{Gal20,FlaLuo20} from the case $a=b=0$. 
    
    In what follows, we focus only on the case of deterministic initial condition. At first glance, a larger $a$ appears to enhance the dissipation in the scaling limit, which might suggest better regularity. In fact, the situation is quite the opposite when performing energy estimates. By a commutator argument similar to that in \cite{CFH19}, we can show in Proposition \ref{GWP-NS-cut-off} that for every $s\neq 2b+a$,
    \begin{equation}\label{intro-higher-estimate}
    	\d\|\omega^N\|_{H^s}^2\lesssim_N \|\omega^N\|_{H^{s+2a}}^2\,\d t + \text{ Martingale part },
    \end{equation}
    which is closed if and only if $a\leq 0$. The estimate \eqref{intro-conservation} and \eqref{intro-higher-estimate} imply that the only closed estimate is the formally conservative quantity $H^{2b+a}$ when $a>0$. That is to say, the pseudo-transport noise does create singularity for $a>0$. In addition, we emphasize that while the energy estimate \eqref{intro-higher-estimate} closes for fixed $N$ as long as $a\leq 0$, \textbf{uniform} a priori estimates for the sequence $\{\omega^N\}$ do not hold beyond the regularity of $L^\infty (0,T; H^{2b+a})$. This is due to the lack of uniform smoothness of $\{\sigma_k\}$. Indeed, by the explicit expression \eqref{theta^N} below, for every $\epsilon>0$, we have
    \begin{equation}\label{no uniform smooth}
    	\sup_{N}\sum_k \|\sigma_k^N\|_{W^{a,\infty}}^2\leq \nu,\quad \sup_{N}\sum_k \|\sigma_k^N\|_{H^{a+\epsilon}}^2 =\infty.
    \end{equation}
    This indicates that any derivative estimate of order higher than $a$ cannot be uniform in $N$. Given that the best uniform bound for $\{\omega^N\}$ is in $L_t^\infty H^{2b+a}$, convergence is restricted to the strictly weaker space $C([0,T];H^{2b+a-\epsilon}) $.
    
    Our main results heavily rely on the above scaling limit, which captures the dissipation enhancement by the pseudo-transport noise. For simplicity, we only give informal statements of delayed blow-up on finite time interval $[0,T]$, where $T>0$ is fixed. We remark that thanks to the small-data global well-posedness of the NS equations, the strong solutions obtained in Theorems \ref{informal-1} (1) and \ref{informal-2} below extend globally in time; cf. Theorems \ref{Delay blow up-NS} and \ref{SMR-delayed blow-up}.
    
    \begin{theorem}[Informal version of Theorem \ref{Delay blow up-NS} and \ref{Delay blow-up Euler}]\label{informal-1}
    	Assume $T>0, 2b+a>\frac{5}{2}$, and fix a bounded ball $\mathcal{B}\subset H^{2b+a}$. The following holds:
    	\begin{itemize}
    		\item [(1)] \emph{Navier-Stokes case} ($\kappa =1$). Let $a\in [\frac{1}{4},\frac{1}{2})$. Then, for every $\epsilon\in (0,1)$ and $\nu>0$, there exists a finite family divergence-free vector fields $\{\sigma_k\}\in \mathcal{S}^\nu$, such that for all divergence-free vector field $u_0\in \mathcal{B}$, the stochastic NS equations \eqref{fluid eq}  admits a unique strong solution up to time $T$ with a probability no less than $1-\epsilon$.
    		\item [(2)] \emph{Euler case} ($\kappa =0$). Let $a\in [-1,0]$. Then, for every $\epsilon\in (0,1)$, there exists $\nu\gg 0$ and a finite family divergence-free vector fields $\{\sigma_k\}\in \mathcal{S}^\nu$, such that for all divergence-free vector field $u_0\in \mathcal{B}$, the stochastic Euler equations \eqref{fluid eq}  admits a unique strong solution up to time $T$ with a probability no less than $1-\epsilon$.
    	\end{itemize}
    \end{theorem}
    Below, we briefly explain the results.
    Choose $\sigma_k=\sigma_k^N$ as in \eqref{theta^N}, fix initial data $u_0\in H^{2b+a}$ and denote the corresponding (local) solution to \eqref{fluid eq} by $u^N$. Formally by scaling limit, we expect that $u^N$ converges to the solution $\bar{u}$ of the deterministic equation:
    \begin{equation}
    	\partial_t \bar{u} +\Pi(\bar{u}\cdot\nabla\bar{u}) = \kappa \Delta \bar{u} - \frac{3\nu}{5}(-\Delta)^{1+a}\bar{u}. 
    \end{equation}
    The factor $\frac{3}{5}$ is due to the Leray projection in $\Lc_k$. The delayed blow-up results in Theorem \ref{informal-1} are based on two key points:
    \begin{enumerate}
    	\item the limit $\bar{u}$ exists globally and satisfies good energy estimates,
    	\item $u^N$ converges to $\bar{u}$ in an appropriate space, thereby providing better estimates for $u^N$.
    \end{enumerate}
    First, (1) always holds when $a \ge \frac{1}{4}$ (corresponding to the Lions exponent \cite{Lio69}), while for $a \in [-1, \frac{1}{4})$, the global regularity of $\bar{u}$ is guaranteed provided that the noise intensity $\nu$ is sufficiently large. This explains why no large-noise assumption on $\nu$ is needed for Theorem \ref{informal-1} (1) to hold. Second, as mentioned at the beginning of Section 1.1, control in a critical space is required at the very least to ensure the global existence of $u^N$; on the other hand, since the scaling limit converges only in $C([0, T]; H^{2b+a-\epsilon})$, it implies that the initial data in $H^{2b+a}$ must be subcritical. In the NS equations, the extra regularity from the viscosity $\kappa \Delta$ allows one to handle the $2a$-order singularity (cf. eq. \eqref{intro-higher-estimate}) induced by the pseudo-transport noise for $a < \frac{1}{2}$, hence guarantees the local well-posedness of $u^N$. In contrast, for the Euler equations, we are restricted to the more regular case, namely $a \le 0$.  For the case where $a$ is too large and the noise becomes overly singular, see Proposition \ref{a>0.5}, \ref{a>0.25}-\ref{a<0.25}.
    
    In the above discussion, the viscosity is used solely to ensure the local well-posedness of $u^N$ and does not affect the scaling limit. Doing $H^{2b+a}$ energy estimates and neglecting the nonlinear term, we have the uniform bounds 
    \begin{equation*}
    	\sup_{N} \Big(\|u^N\|_{C([0,T],H^{2b+a})} + \|u^N\|_{L^2(0,T;H^{2b+a+1})}\Big)<\infty, \quad\Pb\text{-a.s.}
    \end{equation*}
    Agresti \cite{Agr25, Agr26} first extended the uniform estimates for the scaling limit of transport noises to the $L^p$ setting with $p > 2$. Under the assumption $a \in [-1, 0)$, the perturbation argument in \cite{Agr26} can be adapted to our framework. Indeed, for every $2 \le p < \infty$, we have:
    \begin{equation}\label{intro-SMR}
    	\sup_{N} \Big(\E\Big[\|u^N\|_{C([0,T],X_{tr})}^p\Big] + \E\Big[\|u^N\|_{L^p(0,T;H^{2b+a+1})}^p\Big]\Big)<\infty, 
    \end{equation}
    where $X_{tr}=B_{2,p}^{2b+a+1-\frac{2}{p}}$.  The restriction $a\in [-1,0)$ ensures that the pseudo-transport noise is of lower order relative to the viscosity $\kappa \Delta$, allowing a perturbation argument (cf. Proposition \ref{USMR}) to yield estimates \textbf{uniform} in $N$. The borderline case $a=0$ (transport noise) is more delicate: here, the noise is of the same order as the leading operator $\kappa\Delta$ within the framework of stochastic maximal regularity, yet the vector fields $\{\sigma_k^N\}$ become increasingly rough as $N\rightarrow \infty$ according to \eqref{no uniform smooth}, rendering the perturbation argument invalid. In \cite{Agr26}, the author employs the hyperviscosity $-\kappa(-\Delta)^{\gamma}$ with $(\gamma>1)$ to derive the uniform estimate similar to \eqref{intro-SMR}. We refer the reader to \cite[Section 1.3]{Agr26} for a more detailed discussion.
    
   By choosing $p \gg 2$ so that $L^p(0,T; H^{2b+a+1})$ and $C([0,T],X_{tr})$ are subcritical for the 3D NS equations, and by exploiting this uniform estimate in the scaling limit, we arrive at the following results:
    \begin{theorem}[Informal version of Theorem \ref{SMR-finite time-delayed blow-up}]\label{informal-2}
    	Assume $T>0, a\in [-1,0), 2b+a>-\frac{1}{2}$, and fix a bounded ball $\mathcal{B}\subset H^{2b+a+1}$.  Then, for every $\epsilon\in (0,1)$, there exists $\nu\gg 0$ and a finite family divergence-free vector fields $\{\sigma_k\}\in \mathcal{S}^\nu$, such that for all divergence-free vector field $u_0\in \mathcal{B}$, the stochastic NS equations \eqref{fluid eq}  admits a unique strong solution up to time $T$ with a probability no less than $1-\epsilon$.
    \end{theorem}

    \subsection{Organization and notation}
    The remainder of this paper is organized as follows. Section \ref{Scaling-Limit} introduces the pseudo-transport noise and derives its scaling limit for both deterministic and random initial conditions. Section \ref{NS-I} is devoted to the delayed blow-up for the NS equations; notably, the main result (Theorem \ref{Delay blow up-NS}) holds for any $\nu>0$. In contrast, Section \ref{Euler} addresses the Euler equations, where a sufficiently large $\nu$ is required (Theorem \ref{Delay blow-up Euler}). Finally, in Section \ref{NS-II}, we establish uniform stochastic $L^p$ maximal regularity estimates for the more regular regime $a<0$, and proves the corresponding results for the NS equations.
     
    Now we collect some frequently used notations. The symbol $f\lesssim_\alpha g$ for two functions $f$ and $g$ means that there exists a positive
    constant $C_\alpha$ depending on the parameter $\alpha$ such that $f(x)\leq C_\alpha g(x)$ for all $x$. Let $(\Omega,\Fc,(\Fc_t),\Pb)$ be a filtered probability space which satisfies the usual conditions and $\E$ denotes the expectation.
     
    For a measure space $(S,\mu), p \in (1,\infty)$ and a Banach space $X$, we write $L^p(S,\mu;X)$ for the Bochner space of strongly measurable, $p$-integrable $X$ valued functions. Fix $T>0$, we denoted by $W^{1,p}(0,T;X)$ the set of all $f \in L^p(0,T;X)$ such that $f'\in L^p(0,T;X)$ endowed with the natural norm. Moreover, for $\theta\in (0,1)$, the Bessel-potential space $H^{\theta,p}(0,T;X)$ is defined as
    \begin{equation*}
     	H^{\theta,p}(0,T;X) :=  \big[L^p(0,T;X),W^{1,p}(0,T;X)\big]_{\theta},
    \end{equation*}
    where $[\cdot,\cdot]_{\theta}$ denotes the complex interpolation, see \cite{HvNVMW16} for details.
     
    Next, we introduce the function spaces on the torus $\Tb^d=[-\pi,\pi]^{d}$. For a tempered distribution $u \in \mathcal{S}'(\Tb^d)$, let $\hat{u}$ or $\Fc[u]$ denote its Fourier transform, which is defined by
    \begin{equation*}
    	\hat{u}(k) = \int_{\Tb^d} u(x)e_{-k}(x)\,\d x, \quad\text{ for every } k\in \Z_0^d, 
    \end{equation*} 
    where $e_k(x)=(2\pi)^{-\frac{d}{2}}e^{ik\cdot x}$.
    Throughout this paper, we restrict our attention to distributions with zero mean, i.e., $\hat{u}(0) = 0$. Under this convention, for $s \in \mathbb{R}$ and $1 \leq p, q \leq \infty$, we denote by $B_{q,p}^s(\mathbb{T}^d)$ the standard Besov space on the torus. In the special case $p = q = 2$, the Besov space $B_{2,2}^s(\mathbb{T}^d)$ coincides with the usual Sobolev space $H^s(\mathbb{T}^d)$. 
     
    Finally, the Leray projection $\Pi$ is a Fourier multiplier with the symbol
    \begin{equation}\label{Leray}
   	    \widehat{\Pi u}(\xi) = P_{\xi}^{\perp} \hat{u}(\xi), \quad \xi \in \mathbb{Z}_0^d,
    \end{equation}
    where $P_{\xi}^{\perp} = I_d - \frac{\xi \otimes \xi}{|\xi|^2}$ denotes the projection onto the space orthogonal to $\xi$.
    
	\section{Scaling limit for pseudo-transport noise}\label{Scaling-Limit}
	Consider the scalar equation driven by the pseudo-transport noise on d dimensional torus $\Tb^d$:
	\begin{equation}\label{scalar eqn}
		\left\{
		\begin{aligned}
			& \d \omega =\sum_k (-\Delta)^{-b}\big( \sigma_k\cdot\nabla (-\Delta)^{a+b} \omega\big) \circ \d W^k,\\
			& \omega(0)=\omega_0,
		\end{aligned}
		\right.
	\end{equation}
	where $\circ$ denotes the Stratonovich integral, $\{\sigma_k\}$ is a finite family of smooth divergence-free vector field, $\{W^k\}$ is a sequence of standard Brownian motion and $a,b\in\R$ are two parameters. Due to the special structure of the noise, the equation admits a conserved quantity and a corresponding invariant measure. 
	\begin{proposition}\label{conservation by noise}
		At least formally, the following holds:
		\begin{itemize}
			\item[(1). ] Deterministic initial condition: If $\omega_0\in H^{2b+a}$, then for every $t\ge0$, $\omega(t)\in H^{2b+a}$ and $\Vert \omega(t) \Vert_{H^{2b+a}}=\Vert \omega_0 \Vert_{H^{2b+a}}$.
			\item[(2). ] Random initial condition: If Law $(\omega_0)$= $(-\Delta)^{-\frac{a}{2}-b}\mu$, where $\mu$ is the white noise on $\Tb^d$, then $\{\omega(t)\}_{t\ge 0}$ is stationary: for every $t\geq 0$, Law$(\omega(t))$ $=\mu$. 
		\end{itemize}
	\end{proposition}
	
	\begin{proof}
		With out loss of generality, we can assume $b=0$. Indeed, by considering the equation satisfied by $(-\Delta)^{b}\omega$, the problem can be reduced to this special case.
		Since $\Vert \omega \Vert_{H^{a}}^2=\<f,(-\Delta)^{a}f\>$, it is straightforward to check that
		\begin{equation*}
			\d \Vert \omega \Vert_{H^{a}}^2=  2\sum_k \big\<(-\Delta)^{a}\omega, \sigma_k\cdot\nabla (-\Delta)^{a} \omega\big\> \circ \d W^k=0,
		\end{equation*} 
		where the last step follows from $\divg \sigma_k=0$. To prove (2), it suffices to check that $\mu$ is infinitesimally invariant in the sense that for every smooth cylinder functional $F$, we have
		\begin{equation}\label{infinitesimal}
			\int \Lc F(\omega)\,\d \mu(\omega)=0,
		\end{equation}
		where $\Lc$ denotes the generator, which is defined as
		\begin{equation*}
			\Lc F(\omega)=\frac{1}{2}\sum_k \big\<\sigma_k\cdot\nabla(-\Delta)^{a}\omega,D\<\sigma_k\cdot\nabla(-\Delta)^{a}\omega, DF(\omega)\>\big\>.
		\end{equation*} 
		We will denote by $\delta_\mu$ the dual of the Malliavin derivative $D$ in $L^2(\mu)$, then \eqref{infinitesimal} is equivalent to the following:
		\begin{equation*}
			\sum_k\int \big\<\delta_\mu(\sigma_k\cdot\nabla(-\Delta)^{a}\omega), \<\sigma_k\cdot\nabla(-\Delta)^{a}\omega, DF(\omega)\> \big\>=0.
		\end{equation*}
		Arguing exactly as \cite[Section 3]{AHKDeF79}, it can be shown that 
		\begin{equation*}
			\delta_\mu(\sigma_k\cdot\nabla(-\Delta)^{a}\omega)=-\divg_\omega (\sigma_k\cdot\nabla(-\Delta)^{a}\omega)+\<\sigma_k\cdot\nabla(-\Delta)^{a}\omega, (-\Delta)^{a}\omega\>=0,
		\end{equation*}
		which implies that \eqref{infinitesimal} holds. The proof is complete. 
	\end{proof}

	\begin{remark}
		\begin{enumerate}
			\item 	For deterministic initial data $\omega_0 \in H^{2b+a}$, the existence of probabilistically weak solutions for the given initial data follows from the a priori estimates obtained in the above proposition together with a standard Galerkin approximation argument. Besides, with some assumptions on the noise,  pathwise uniqueness can be established by adopting the strategy from Barbato et al. \cite{BBF14}, see Appendix \ref{linear-unique} for details. Consequently, the Yamada-Watanabe theorem ensures the existence of a unique probabilistically strong solution. 
			\item 	For random initial data  Law ($\omega_0$)= $(-\Delta)^{-\frac{a}{2}-b}\mu$, the weak existence of a stationary solution to \eqref{scalar eqn} follows a similar compactness argument as in \cite[Theorem 2.2.3]{ASC90}. Here, a progressively measurable process $\omega$ is stationary if that  for every $t\geq 0$, Law($\omega(t)$)$=\mu$. 
		\end{enumerate}
	\end{remark}

	Now we introduce the scaling limit for the pseudo-transport noise. Let 
	\begin{equation*}
		e_k(x) =\sqrt{2}(2\pi)^{-\frac{d}{2}}
		\begin{cases}
			\cos( k\cdot x), & k \in \Z^d_+, \\
			\sin( k\cdot x), & k \in \Z^d_-,
		\end{cases}
		\quad x \in \mathbb{T}^d,
	\end{equation*} 
	where $\Z_+^d$ and $\Z_-^d$ form a partition of $\Z_0^d$ with $\Z_-^d = -\Z_+^d$. Then $\{e_k: k\in \Z_0^d\}$ constitute a complete orthonormal basis of $L^2(\Tb^d)$, the space of $\R$-valued square integrable functions with zero average. For each $k\in \Z_-^d$, let $\{a_{k,i}: \,i=1,\cdots, d-1\}$ be an orthonormal basis of $k^\perp$. Now for $\nu>0,\;\gamma\in [0,d/2]$, define 
	\begin{equation}\label{theta^N}
		\begin{split}
			\theta_k^N=|k|^{-a-\gamma}\mathbb{I}_{\{|k|\leq N\}},\;
			\sigma_{k,i}^N (x)=\frac{\sqrt{C_d\nu}}{|\theta^N|_{h^{a}}} \theta_k^Na_{k,i}e_k(x),\quad N\in \N_+,\;k\in\Z_0^d,  
		\end{split}
	\end{equation}
	where $C_d= (2\pi)^d\frac{2d}{d-1}$ and the normalization constant is
	\begin{equation*}
		|\theta^N|_{h^{a}}:=\Big(\sum_k |k|^{2a} (\theta_k^N)^2\Big)^{1/2}\sim
		\begin{cases}
			N^{d/2-\gamma}, & \gamma\in [0,d/2),\\
			\sqrt{log N}, & \gamma=d/2.
		\end{cases}
	\end{equation*} 
	The covariance matrix $Q^N(x,y):=\sum_{k,i} \sigma_{k,i}^N (x)\otimes \sigma_{k,i}^N (y)$ is homogeneous and has expression
	\begin{equation}\label{Q^N}
		Q^N(x,y)=Q^N(x-y)=(2\pi)^{-d} \frac{C_d\nu}{|\theta^N|_{h^{a}}^2}\sum_{k\in \Z_0^d} (\theta_k^N)^2 \Big(I_d-\frac{k\otimes k}{|k|^2}\Big) e^{ ik (x-y)}.
	\end{equation}
	In particular, as $N\rightarrow 0$, we have
	\begin{equation}\label{Q^N-property}
		\sup\limits_{k\in\Z_0^d}\big|\widehat{Q^N}(k)|k|^{2a}\big|\lesssim_{d} \frac{\nu}{|\theta^N|_{h^{a}}^2}\sup\limits_{k\in\Z_0^d} (\theta_{k}^N)^2|k|^{2a}= \frac{\nu}{|\theta^N|_{h^{a}}^2}\rightarrow 0.
	\end{equation}
	We will see that all conclusions in this paper depend only on the form of the covariance matrix $Q^N$. Therefore, in what follows we omit the subscript $i$ for simplicity, and continue to denote the vector fields by $\sigma_k^N$.
	
	\begin{remark}
		In the special case $a=0$, the covariance matrix $Q^N$ coincides with the covariance appearing in the standard scaling limit in \cite{FlaLuo20,LuoZhu21,Gal20}.
		There are also other choices of the sequence $\{\theta^N\}$ such that the results of this paper remain valid. For example, for $\gamma\in \R$, we can take $\theta_k^N=|k|^{-a-\gamma}\mathbb{I}_{\{N\leq|k|\leq 2N\}}$. Indeed, one can verify that $\big\{(-\Delta)^{a/2}\sigma_{k,i}^N\big\}$ coincides with the one considered in \cite{FlaLuo20, Gal20} for the scaling limit of transport noise.
	\end{remark}
	
	Now we can give the precise statement of the scaling limit.
	\begin{theorem}\label{scaling limit}
		Let $\omega^N$ be the solution to \eqref{scalar eqn} with $\sigma_k=\sigma_k^N$, we have the following:
		\begin{itemize}
			\item[(1).] Deterministic initial condition: If $\omega_0\in H^{2b+a}$, then $\omega^N$ converges in law, in $C([0,T];\mathcal{S}')$, to the solution of the (fractional) heat equation:
			\begin{equation*}
				\left\{
				\begin{aligned}
					& \partial_t \omega =-\nu(-\Delta)^{1+a}\omega,\\
					& \omega(0)=\omega_0.
				\end{aligned}
				\right.
			\end{equation*}
			\item[(2).] Random initial condition: If Law ($\omega_0$)= $(-\Delta)^{-\frac{a}{2}-b}\mu$, then $\omega^N$ converges in law, in $C([0,T];\mathcal{S}')$, to the stationary infinite dimensional Ornstein-Uhlenbeck process:
			\begin{equation}\label{sto-limit process}
				\left\{
				\begin{aligned}
					& \d \omega =-\nu(-\Delta)^{1+a}\omega + \sqrt{2\nu}(-\Delta)^{1/2-b}\d W,\\
					& \omega(0)=\omega_0,
				\end{aligned}
				\right.
			\end{equation}
			where $\d W$ is the space-time white noise.
		\end{itemize}
	\end{theorem}
	
	\begin{remark}
		It is clear that the scaling limit results for transport noises in \cite{FlaLuo20,Gal20} are special case $a=b=0$ in Theorem \ref{scaling limit}. Moreover, we observe that the parameter $a$ modifies the dissipative part of the limit, while $b$ modifies the fluctuation part.
	\end{remark}
	
	The It\^o's form of \eqref{scalar eqn} reads
	\begin{equation*}
		\begin{split}
			\d \omega^N &=\sum_k (-\Delta)^{-b}\big( \sigma_k^N\cdot\nabla (-\Delta)^{a+b} \omega^N\big) \, \d W^k\\ &\quad+ \frac{1}{2}\sum_k (-\Delta)^{-b}\Big( \sigma_k^N\cdot\nabla (-\Delta)^{a} \big(\sigma_k^N\cdot\nabla (-\Delta)^{a+b}\omega^N\big)\Big)\,\d t \\
			&=:\d M^{N} + S^N(\omega^N)\,\d t.
		\end{split}
	\end{equation*}
	The  proof of Theorem \eqref{scaling limit} relies on analyzing the limiting behavior of the It\^o-Stratonovich corrections $S^N$ and the martingale part $M^N$ under different regularity assumptions on the underlying process. As noted in the proof of Proposition \ref{conservation by noise}, we may assume without loss of generality that $b=0$. The following lemma  describes the asymptotic behavior of the Itô-Stratonovich corrections.
	
	\begin{lemma}\label{Ito corrections limit}
		For every $r\in \R,\;\varphi\in C^\infty(\Tb^d)$, the following limit holds in $H^r(\Tb^d)$:
		\begin{equation*}
			\lim\limits_{N\rightarrow\infty}S^N(\varphi)=-\nu(-\Delta)^{1+a}\varphi.
		\end{equation*}
	\end{lemma}
	\begin{proof}
		Taking Fourier transform, we obtain, for every $\xi\in \Z_0^d$,
		\begin{equation*}
			\begin{split}
				2\widehat{S^{N}(\varphi)}(\xi)
				&=\sum_k\mathcal{F} \big[\sigma_k^N\cdot\nabla(-\Delta)^{a}(\sigma_k^N\cdot\nabla (-\Delta)^a\varphi)\big](\xi) \\
				&=(2\pi)^{-d}\sum_{k,\eta,\zeta} \widehat{\sigma_k^N}(\xi-\eta)\cdot i\eta\, |\eta|^{2a} \widehat{\sigma_k^N} (\eta-\zeta)\cdot i\zeta\, |\xi|^{2a}\hat{\varphi} (\zeta)
			\end{split}
		\end{equation*}
		Following \cite[Proposition 1.3]{LuoXieZhao24}\footnote{There, the factor $(2\pi)^{d/2}$ is missing; it comes from the convolution formula for the Fourier transform
			$\widehat{f * g} = (2\pi)^{d/2}\,\hat f\,\hat g .$}, we have
		\begin{equation*}
			\sum_{k} \widehat{\sigma_k^N}(\xi-\eta)\otimes \widehat{\sigma_k^N} (\eta-\zeta)=(2\pi)^{d/2}\widehat{Q^N}(\xi-\eta)\mathbb{I}_{\xi=\zeta}.
		\end{equation*}
		Substituting into above expression, we get
		\begin{align*}
			2\widehat{S_N(\varphi)}(\xi)
			&=-(2\pi)^{-d/2}\sum_{\eta\in \Z_0^d} \xi^T\widehat{Q^N}(\xi-\eta)\eta\, |\eta|^{2a}\: |\xi|^{2a}\hat{\varphi}(\xi)  \\
			&=-(2\pi)^{-d/2}\sum_{\eta\in \Z_0^d} \xi^T\widehat{Q^N}(\xi-\eta)\xi\, |\eta|^{2a}\: |\xi|^{2a}\hat{\varphi}(\xi)  \\
			&=-(2\pi)^{-d}\frac{C_d\nu}{|\theta^N|_{h^{a}}^2}\xi^T\bigg(\sum_{\eta\neq \xi}(\theta_{\eta}^N)^2 \Big(I_2-\frac{\eta\otimes\eta}{|\eta|^2}\Big)|\xi-\eta|^{2a}\bigg)\xi\:|\xi|^{2a}\hat{\varphi}(\xi),
		\end{align*}
		where we have used the specific expression of $Q^N$ in \eqref{Q^N}. On the other hand, we observe that
		\begin{equation*}
			-2\nu\widehat{(-\Delta)^{1+a}\varphi}(\xi)=-(2\pi)^{-d}\frac{C_d\nu}{|\theta^N|_{h^{a}}^2}\xi^T\bigg(\sum_{\eta\in \Z_0^d}(\theta_{\eta}^N)^2 \Big(I_2-\frac{\eta\otimes\eta}{|\eta|^2}\Big)|\eta|^{2a}\bigg)\xi\:|\xi|^{2a}\hat{\varphi}(\xi).
		\end{equation*}
		To show that $S^N(\varphi)$ and $-\nu(-\Delta)^{1+a}\varphi$ are close, it suffices to show that for every $\xi\in \Z_0^d$,
		\begin{equation*}
			R^N:=\frac{1}{|\theta^N|_{h^{a}}^2}\bigg| \sum_{\eta\neq \xi}(\theta_{\eta}^N)^2 \Big(I_2-\frac{\eta\otimes\eta}{|\eta|^2}\Big)\big(|\xi-\eta|^{2a}- |\eta|^{2a}\big)\bigg|\rightarrow 0, \quad \text{ as } N\rightarrow\infty.
		\end{equation*}
		For every $M>0$, $R^N$ is bounded by
		\begin{align*}
			R^N&\leq\frac{1}{|\theta^N|_{h^{a}}^2} \sum_{\eta\neq \xi}(\theta_{\eta}^N)^2 \big| |\xi-\eta|^{2a}- |\eta|^{2a} \big|  \\
			&= \frac{1}{|\theta^N|_{h^{a}}^2} \sum_{\eta\neq \xi} (\theta_{\eta}^N)^2|\eta|^{2a}\, \Big|\frac{|\xi-\eta|^{2a}}{|\eta|^{2a}}-1\Big| \\
			&\leq \frac{1}{|\theta^N|_{h^{a}}^2} \sum_{|\eta|\leq M}\cdots \quad+ \frac{1}{|\theta^N|_{h^{a}}^2} \sum_{|\eta|\geq M}\cdots=:R_{\leq M}^N+R_{\geq M}^N.
		\end{align*}
		We have the estimates
		\begin{equation*}
			R_{\leq M}^N\leq  \frac{\sup\limits_{|\eta|\leq M} (\theta_{\eta}^N)^2|\eta|^{2a}}{|\theta^N|_{h^{a}}^2}\sum_{|\eta|\leq M}\Big|\frac{|\xi-\eta|^{2a}}{|\eta|^{2a}}-1\Big|,\quad R_{\geq M}^N \leq \sup\limits_{|\eta|\geq M}\Big|\frac{|\xi-\eta|^{2a}}{|\eta|^{2a}}-1\Big|.
		\end{equation*}
		BY \eqref{Q^N-property}, it is obvious that for every $M>0, \lim\limits_{N\rightarrow\infty} R_{\leq M}^N=0$. Since the second term $R_{\geq M}^N$ converges to $0$ uniformly in $N$ as $M \to \infty$, we obtain $\lim\limits_{N\rightarrow\infty} R^N=0$. 
		
		Now we have proved that, for every $\xi\in \Z_0^d$, 
		\begin{equation*}
			\lim\limits_{N\rightarrow\infty}\big|\mathcal{F}\big[(S^N(\varphi))+\nu(-\Delta)^{1+a}\varphi\big] (\xi)\big|=0.
		\end{equation*}
		Hence the desired convergence follows from the fast decay at infinity of the Fourier transform of smooth functions. Indeed, for every $M>0$, we have
		\begin{align*}
			&\sum_{|\xi|\geq M} \big|\mathcal{F}\big[(S^N(\varphi))+\nu(-\Delta)^{1+a}\varphi\big] (\xi)\big|^2|\xi|^{2r}\\
			&\lesssim_{d,\nu}  \sum_{|\xi|\geq M}|\xi|^{4a+4+2r} |\hat{\varphi}(\xi)|^2\sup\limits_{\eta}\Big|\frac{|\xi-\eta|^{2a}}{|\eta|^{2a}}-1\Big|^2  \\
			&\lesssim_{d,\nu,a} \sum_{|\xi|\geq M}|\xi|^{4a+4|a|+4+2r} |\hat{\varphi}(\xi)|^2,
		\end{align*}
		where the last step follows from the fact that for every $\xi\in \Z_0^d$,
		\begin{equation}\label{intermidiate}
			\sup\limits_{\eta}\Big|\frac{|\xi-\eta|^{2a}}{|\eta|^{2a}}-1\Big|\lesssim_a |\xi|^{2|a|}.
		\end{equation}
		Therefore we obtain
		\begin{align*}
			\Big\| S^N(\varphi) + \nu(-\Delta^{1+a}\varphi) \Big\|_{H^r}^2&\leq \sum_{|\xi|\geq M} \big|\mathcal{F}\big[(S^N(\varphi))+\nu(-\Delta)^{1+a}\varphi\big] (\xi)\big|^2|\xi|^{2r}\\ &\quad+ C(d,\nu,a)\sum_{|\xi|\geq M}|\xi|^{4a+4|a|+4+2r} |\hat{\varphi}(\xi)|^2.
		\end{align*}
		The assertions follows from the analysis similar to that of $R^N$.
	\end{proof}
	
	For deterministic initial data $\omega_0\in H^a$, the martingale part $M^N$ vanishes in the limit.
	\begin{lemma}\label{det-martingale-limit}
		For every $\varphi\in C^\infty(\Tb^d)$, write 
		\begin{equation*}
			M_t^N(\varphi) = \sum_k \int_{0}^{t}\big\< \sigma_k^N\cdot\nabla (-\Delta)^{a} \omega^N,\varphi\big\> \, \d W^k,\quad \forall t\geq0.
		\end{equation*}
		Then for every $T>0$, $\big|\sup\limits_{t\in [0,T]}M_t^N(\varphi)\big|\rightarrow 0$ in probability as $N\rightarrow \infty$.
	\end{lemma}
	
	\begin{proof}
		It suffices to check $\E\Big[\big|M_T^N(\varphi)\big|^2\Big]\rightarrow 0$ as $N\rightarrow 0$. By It\^o's isometry, we have
		\begin{equation}\label{det-ini-martingale}
			\E\Big[\big|M_T^N(\varphi)\big|^2\Big] = \sum_k \int_{0}^{T} \big|\big\<  (-\Delta)^{a} \omega^N,\sigma_k^N\cdot\nabla\varphi\big\>\big|^2\,\d t. 
		\end{equation}
		Taking Fourier transform, it follows that
		\begin{equation*}
			\begin{split}
				&(2\pi)^{d}\sum_k\big|\big\<  (-\Delta)^{a} \omega^N,\sigma_k^N\cdot\nabla\varphi\big\>\big|^2  \\ 
				&= -\sum_{k,\xi,\xi',\eta,\eta'}\widehat{\sigma_k^N}(\xi-\eta)\cdot \eta \hat{\varphi}(\eta)|\xi|^{2a}\overline{\widehat{\omega^N}}(\xi)\, \widehat{\sigma_k^N}(\xi'-\eta')\cdot \eta' \hat{\varphi}(\eta')|\xi'|^{2a}\overline{\widehat{\omega^N}}(\xi')\\
				&=-\sum_{\xi+\xi'=\eta+\eta'}\eta^T \widehat{Q^N}(\xi-\eta)\eta'|\xi|^{2a}\overline{\widehat{\omega^N}}(\xi)|\xi'|^{2a}\overline{\widehat{\omega^N}}(\xi')\hat{\varphi}(\eta)\hat{\varphi}(\eta').
			\end{split}
		\end{equation*}
		For fixed $\eta,\eta\in \Z_0^d$, we have the estimate by H\"older's inequality:
		\begin{align*}
			&\Big|\sum_{\xi+\xi'=\eta+\eta'} \widehat{Q^N}(\xi-\eta)|\xi|^{2a}\overline{\widehat{\omega^N}}(\xi)|\xi'|^{2a}\overline{\widehat{\omega^N}}(\xi')\Big|  \\
			&= \bigg|\sum_{\xi+\xi'=\eta+\eta'} \widehat{Q^N}(\xi-\eta)|\xi-\eta|^{2a}|\xi|^{a}\overline{\widehat{\omega^N}}(\xi)|\xi'|^{a}\overline{\widehat{\omega^N}}(\xi')\times \frac{|\xi|^{a}|\xi'|^{a}}{|\xi-\eta|^{2a}}\bigg| \\
			&\leq \|\omega^N\|_{H^a}^2\sup\limits_{\xi}\big|\widehat{Q^N}(\xi)|\xi|^{2a}\big| \sup\limits_{\xi}\frac{|\xi|^{a}|\xi-\eta-\eta'|^{a}}{|\xi-\eta|^{2a}} \\
			&\lesssim_a \|\omega_0\|_{H^a}^2 |\eta|^{|a|}|\eta'|^{|a|}\sup\limits_{\xi}\big|\widehat{Q^N}(\xi)|\xi|^{2a}\big|,
		\end{align*} 
		where the last step follows from Proposition \ref{conservation by noise} (which implies $\|\omega^N\|_{H^a}=\|\omega_0\|_{H^a},\Pb\text{-a.s.}$) and the estimate \eqref{intermidiate}:
		\begin{equation*}
			\sup\limits_{\xi}\frac{|\xi|^{a}|\xi-\eta-\eta'|^{a}}{|\xi-\eta|^{2a}}\leq\sup\limits_{\xi}\frac{|\xi|^{a}}{|\xi-\eta|^{a}}\sup\limits_{\xi}\frac{|\xi-\eta-\eta'|^{a}}{|\xi-\eta|^{a}} \lesssim_a |\eta|^{|a|}|\eta'|^{|a|}.
		\end{equation*}
		Substituting all above into \eqref{det-ini-martingale}, we arrive at
		\begin{equation*}
			\E\Big[\big|M_T^N(\varphi)\big|^2\Big] \lesssim_{d,a}T\|\omega_0\|_{H^a}^2\Big(\sum_{\eta}|\hat{\varphi}(\eta)|\,|\eta|^{1+|a|}\Big)^2 \sup\limits_{\xi}\big|\widehat{Q^N}(\xi)\big|\,|\xi|^{2a}.
		\end{equation*}
		Since $\sup\limits_{\xi}\big|\widehat{Q^N}(\xi)|\xi|^{2a}\big|\rightarrow 0$ as $N\rightarrow 0$ by \eqref{Q^N-property}, the proof is thus complete.
	\end{proof}
	
	For random initial data Law($\omega_0$)=$(-\Delta)^{-\frac{a}{2}}\mu$, the quadratic variation of the martingale part $M^N$ will converge to its expectation.
	
	\begin{lemma}\label{sto-martingale-limit}
		For every $\varphi_1,\varphi_2\in C^\infty(\Tb^d)$, let the quadratic variation be given by
		\begin{align*}
			\d\,\<M^N(\varphi_1),M^N(\varphi_2)\>&=\sum_k \big\< \sigma_k^N\cdot\nabla (-\Delta)^{a} \omega^N,\varphi_1\big\>\big\< \sigma_k^N\cdot\nabla (-\Delta)^{a} \omega^N,\varphi_2\big\>\,\d t\\
			&=:I^N_{\varphi_1,\varphi_2}\,\d t.
		\end{align*} 
		Then for every $1\leq p<\infty$, $I^N_{\varphi_1,\varphi_2}\rightarrow -2\nu\<\varphi_1,\Delta\varphi_2\>$ in $L^p(\Omega)$ as $N\rightarrow \infty$.
	\end{lemma}
	
	\begin{proof}
		Since $I^N_{\varphi_1,\varphi_2}$ belongs to the first two Wiener chaos, we only need to prove the special case $p=2$.
		For simplicity of notation, we set $\psi^N=(-\Delta)^a\omega^N$.
		On the one hand, taking expectation, we get
		\begin{equation*}
			\begin{split}
				\E\big[I^N_{\varphi_1,\varphi_2}\big] 
				&= \sum_k \E\big[\<\sigma_k^N\cdot\nabla\varphi_1,\psi^N\> \<\sigma_k^N\cdot\nabla\varphi_2,\psi^N\>\big]\\
				&=  \sum_k \big\< \sigma_k^N\cdot\nabla\varphi_1,(-\Delta)^{a}(\sigma_k^N\cdot\nabla\varphi_2) \big\> \\
				&= -\sum_k \big\<\varphi_1, \sigma_k^N\cdot\nabla (-\Delta)^{a}(\sigma_k^N\cdot\nabla\varphi_2)\big\>= -2\big\< \varphi_1, S^{N}((-\Delta)^{-a}\varphi_2)\big\>.
			\end{split}
		\end{equation*}
		By Lemma \ref{Ito corrections limit}, we have
		\begin{equation*}
			\E\big[I^N_{\varphi_1,\varphi_2}\big] \rightarrow -2\nu\<\varphi_1,\Delta\varphi_2\>.
		\end{equation*}
		On the other hand, we have
		\begin{equation*}
			\begin{split}
				&\quad\E\big[\big(I^N_{\varphi_1,\varphi_2}\big)^2\big]\\
				&=\sum_{k_1,k_2}\E\big[\<\sigma_{k_1}^N\cdot\nabla\varphi_1,\psi^N\> \<\sigma_{k_1}^N\cdot\nabla\varphi_2,\psi^N\>
				\<\sigma_{k_2}^N\cdot\nabla\varphi_1,\psi^N\>\<\sigma_{k_2}^N\cdot\nabla\varphi_2,\psi^N\>\big].
			\end{split} 
		\end{equation*}
		By the Isserlis-Wick theorem,
		\begin{equation*}
			\begin{split}
				&\quad\E\big[\big(I^N_{\varphi_1,\varphi_2}\big)^2\big]
				\\
				&=\sum_{k_1,k_2} \E\big[\<\sigma_{k_1}^N\cdot\nabla\varphi_1,\psi^N\> \<\sigma_{k_1}^N\cdot\nabla\varphi_2,\psi^N\>\big]
				\E\big[\<\sigma_{k_2}^N\cdot\nabla\varphi_1,\psi^N\>\<\sigma_{k_2}^N\cdot\nabla\varphi_2,\psi^N\>\big] \\
				&\quad + \sum_{k_1,k_2} \E\big[\<\sigma_{k_1}^N\cdot\nabla\varphi_1,\psi^N\> \<\sigma_{k_2}^N\cdot\nabla\varphi_1,\psi^N\>\big]
				\E\big[\<\sigma_{k_1}^N\cdot\nabla\varphi_2,\psi^N\>\<\sigma_{k_2}^N\cdot\nabla\varphi_2,\psi^N\>\big] \\
				&\quad + \sum_{k_1,k_2} \E\big[\<\sigma_{k_1}^N\cdot\nabla\varphi_1,\psi^N\> \<\sigma_{k_2}^N\cdot\nabla\varphi_2,\psi^N\>\big]
				\E\big[\<\sigma_{k_1}^N\cdot\nabla\varphi_2,\psi^N\>\<\sigma_{k_2}^N\cdot\nabla\varphi_1,\psi^N\>\big] \\
				&=J_1+J_2+J_3.
			\end{split}
		\end{equation*}
		It is obvious that 
		\begin{equation*}
			J_1=\Big(\sum_k \E\big[\<\sigma_k^N\cdot\nabla\varphi_1,\psi^N\> \<\sigma_k^N\cdot\nabla\varphi_2,\psi^N\>\big]\Big)^2= \Big(\E\big[I^N_{\varphi_1,\varphi_2}\big]\Big)^2.
		\end{equation*}
		Hence we arrive at
		\begin{equation*}
			\mathrm{Var}(I^N_{\varphi_1,\varphi_2})= \E\big[\big(I^N_{\varphi_1,\varphi_2}\big)^2\big]-\Big(\E\big[I^N_{\varphi_1,\varphi_2}\big]\Big)^2=J_2+J_3.
		\end{equation*}
		Our task now is to show that both $J_2$ and $J_3$ vanish in the limit. Recalling that $Law(\omega^N_t)=(-\Delta)^{-\frac{a}{2}}\mu$ for every $t\geq 0,\; N\geq0$ and $\psi^N=(-\Delta)^{a}\omega^N$, we get
		\begin{align*}
			J_2&=\sum_{k_1,k_2}\big\<\sigma_{k_1}^N\cdot\nabla\varphi_1,(-\Delta)^{a}\big(\sigma_{k_2}^N\cdot\nabla\varphi_1\big)\big\> \big\<\sigma_{k_1}^N\cdot\nabla\varphi_2,(-\Delta)^{a}\big(\sigma_{k_2}^N\cdot\nabla\varphi_2\big)\big\> \\
			&=(2\pi)^{-2d}\sum_{\substack{k_1,k_2,\\\xi,\eta,\zeta,\xi',\eta',\zeta'}} \widehat{\sigma_{k_1}^N}(\xi-\eta)\cdot i\eta\,\hat{\varphi}_1(\eta)\,|\xi|^{2a}
			\overline{\widehat{\sigma_{k_2}^N}(\xi-\zeta)\cdot i\zeta\,\hat{\varphi}_1(\zeta)} \\
			&\quad \times \widehat{\sigma_{k_1}^N}(\xi'-\eta')\cdot i\eta'\,\hat{\varphi}_2(\eta)\,(2\pi|\xi'|)^{2a}
			\overline{\widehat{\sigma_{k_2}^N}(\xi'-\zeta')\cdot i\zeta'\,\hat{\varphi}_2(\zeta')}.
		\end{align*}
		Following \cite[Proposition 1.3]{LuoXieZhao24} again, we have
		\begin{equation*}
			\begin{split}
				\sum_{k_1} \widehat{\sigma_{k_1}^N}(\xi-\eta)\otimes \widehat{\sigma_{k_1}^N}(\xi'-\eta') =(2\pi)^{d/2}\widehat{Q^N}(\xi-\eta)\mathbb{I}_{\xi+\xi'=\eta+\eta'},\\
				\sum_{k_2} \widehat{\sigma_{k_2}^N}(\xi-\zeta)\otimes \widehat{\sigma_{k_2}^N}(\xi'-\zeta') =(2\pi)^{d/2}\widehat{Q^N}(\xi-\zeta)\mathbb{I}_{\xi+\xi'=\zeta+\zeta'}.
			\end{split}
		\end{equation*}
		Hence we arrive at
		\begin{equation*}
			\begin{split}
				J_2=\sum_{\substack{\xi+\xi'\\=\eta+\eta'\\=\zeta+\zeta'}}
				\eta^T \widehat{Q^N}(\xi-\eta)\eta'\; \zeta^T \widehat{Q^N}(\xi-\zeta)\zeta'\; |\xi|^{2a}|\xi'|^{2a} \hat{\varphi}_1(\eta)\overline{\hat{\varphi}_1}(\zeta)\hat{\varphi}_2(\eta')\overline{\hat{\varphi}_2}(\zeta')
			\end{split}
		\end{equation*}
		Similar to the proof of Lemma \ref{Ito corrections limit}, we can assume that $\varphi_1=e_m$ and $\varphi_2=e_n$ for some $n,m\in \Z_0^d$, then the above equation reduces to
		\begin{align*}
			J_2=\sum_{\xi+\xi'=m+n}m^T\widehat{Q^N}(\xi-m)n\;m^T\widehat{Q^N}(\xi-m)n\; |\xi|^{2a}|\xi'|^{2a}.
		\end{align*}
		Recalling the definition of $Q^N$ in \eqref{Q^N}, we obtain
		\begin{align*}
			|J_2|&\lesssim_{m,n}\frac{1}{|\theta^N|_{h^{a}}^4}\sum_{\xi} |\theta_{\xi-m}^N|^{4}\;|\xi|^{2a}|m+n-\xi|^{2a} \\
			&\leq \frac{1}{|\theta^N|_{h^{a}}^4}\sup_{\xi}|\theta_{\xi-m}^N|^2|\xi|^{2a}\sum_{\xi} |\theta_{\xi-m}^N|^{2}\;|m+n-\xi|^{2a}\\
			&\lesssim_{m,n} \dfrac{1}{|\theta^N|_{h^{a}}^2}\sup_{\xi}|\theta_{\xi}^N|^2|\xi|^{2a}\rightarrow 0,\quad \text{ as } N \rightarrow \infty.
		\end{align*}
		The term $J_3$ can be handled in muh the same way. Now we have proved that, as $N\rightarrow\infty$,
		\begin{align*}
			\E\big[I^N_{\varphi_1,\varphi_2}\big]\rightarrow -2\nu\<\varphi_1,\Delta \varphi_2\>\quad \text{and}\quad
			\mathrm{Var}(I^N_{\varphi_1,\varphi_2})\rightarrow 0,
		\end{align*}
		which implies $I^N_{\varphi_1,\varphi_2}\rightarrow -2\nu\<\varphi_1, \Delta\varphi_2\>$ in $L^2(\Omega)$.
	\end{proof}
	
	With Lemmas \ref{Ito corrections limit}-\ref{sto-martingale-limit} in hand, we are now ready to prove Theorem \ref{scaling limit}.
	\begin{proof}[Proof of Theorem \ref{scaling limit}]
		Without loss of generality, we can always set $b=0$. By Mitoma's criterion \cite{Mit83}, the assertion follows if we can show that for every $\varphi\in C^\infty(\Tb^d)$, the sequence $\{\<\omega^N,\varphi\>\}$ is tight in $C([0,T];\R)$ and the corresponding convergence holds.
		\begin{enumerate}
			\item For every $\varphi\in C^\infty(\Tb^d)$, by integration by parts, we have
			\begin{equation*}
				\d \<\omega^N,\varphi\> = \d M_t^N(\varphi) + \<\omega^N,(-\Delta)^aS^N(-\Delta)^{-a}\varphi\>.
			\end{equation*}
			Combining Proposition \ref{conservation by noise}, Lemma \ref{Ito corrections limit}-\ref{det-martingale-limit}, the proof is a standard application of the compactness method and  completely analogous to that of \cite[Theorem 2.3 ]{FGL21}, and we omit the details.
			\item For every smooth cylinder functional $F$, recall that the generator $\Lc^N$ of $\omega^N$ is defined as 
			\begin{equation*}
				\begin{split}
					\Lc^N F(\omega) &= \frac{1}{2}\sum_k \big\<\sigma_k^N\cdot\nabla(-\Delta)^{a}\omega,D\<\sigma_k^N\cdot\nabla(-\Delta)^{a}\omega, DF(\omega)\>\big\> \\
					&= \frac{1}{2}\sum_k\Big\<  \sigma_k^N\cdot\nabla (-\Delta)^{a} \big(\sigma_k^N\cdot\nabla (-\Delta)^{a}\omega\big),DF(\omega)\Big\> \\
					&\quad +\frac{1}{2}\sum_k  D^2F(\omega)\Big[\sigma_k^N\cdot\nabla (-\Delta)^{a}\omega,\sigma_k^N\cdot\nabla (-\Delta)^{a}\omega\Big] \\
					&=\big\<S^N(\omega),DF(\omega)\big\> + \frac{1}{2}\sum_{k,l} I_{e_k,e_l}^ND_{e_ke_l}^2F(\omega),
				\end{split}
			\end{equation*}
			where $I_{e_k,e_l}^N:=\sum_k \big\< \sigma_k^N\cdot\nabla (-\Delta)^{a} \omega^N,\varphi_1\big\>\big\< \sigma_k^N\cdot\nabla (-\Delta)^{a} \omega^N,\varphi_1\big\>$. We observe that by Lemma \ref{Ito corrections limit} and \ref{sto-martingale-limit}, formally taking the limit $N\rightarrow \infty$, we obtain
			\begin{equation*}
				\begin{split}
					\Lc^N F(\omega) &\rightarrow -\nu\big\<(-\Delta)^{1+a}\omega,DF(\omega)\big\> - \sum_{k,l}\nu \<e_k,\Delta e_l\> D_{e_ke_l}^2F(\omega) \\
					&=-\nu\big\<(-\Delta)^{1+a}\omega,DF(\omega)\big\> + \nu D^2F(\omega)\big[(-\Delta)^{1/2}\dot{W},(-\Delta)^{1/2}\dot{W}\big],
				\end{split}
			\end{equation*}
			which is exactly the generator $\Lc^\infty$ of the limit process \eqref{sto-limit process}. The above heuristic argument can be made rigorous by following the approach of \cite[Section 2]{FlaLuo20}, where a similar result is proved using a tightness argument, and we omit the details.\qedhere
		\end{enumerate}
	\end{proof}
	
	\begin{remark}
		From the above proof, it is evident that the results of Theorem \ref{scaling limit} can be extended to a more general setting. Specifically, we assume the following conditions on the operators:
		\begin{enumerate}
			\item $A$ is a strictly positive Fourier multiplier that acts as a continuous linear operator on $\mathcal{S}'(\Tb^d)$.
			\item $B$ is an invertible continuous linear operator on $\mathcal{S}'(\Tb^d)$.
		\end{enumerate}
		Under these assumptions, the vector fields $\sigma_k^N$ are defined analogously to \eqref{theta^N}, with the weight $|k|^{-a}$ replaced by $\hat{A}(k)^{-1/2}$, where $\{\hat{A}(k)\}_{k \in \Z_0^d}$ denotes the Fourier coefficients of $A$.
	    Let $\omega^N$ be the weak solution to the equation
	    \begin{equation*}
	    	\left\{
	    	\begin{aligned}
	    		& \d \omega^N  + \sum_k B^{-1}\big(\sigma_k^N\cdot\nabla AB\omega^N\big)\,\circ\d W^k = 0,\\
	    		& \omega^N(0)=\omega_0.
	    	\end{aligned}
	    	\right.
	    \end{equation*}
        The following results hold:
       \begin{enumerate}
       \item If $A^{\frac{1}{2}}B\omega_0\in L^2$, then $\omega^N$ converges in law, in $C([0,T];\mathcal{S}')$ to $\omega$, which is the unique solution to $\omega$ solving 
       \begin{equation*}
           \left\{
       	   \begin{aligned}
       	       & \partial_t \omega =\nu B^{-1}\Delta AB\omega,\\
       		   & \omega(0)=\omega_0.
       	   \end{aligned}
       	   \right.
       \end{equation*}
       \item If $\mu:=$ Law($\omega_0$) is a centered Gaussian measure whose covariance is given by
       \begin{equation*}
       	    \int \<f,\omega\>\,\<g,\omega\>\, \mu(\d\omega) = \big\<f,(BAB^*)^{-1}g\big\>, \quad \text{ for every } f,g\in C^\infty,
       \end{equation*}
       then the stationary solution $\omega^N$ converges in law, in $C([0,T];\mathcal{S}')$ to a stationary process $\omega$, which is the unique solution to
       \begin{equation*}
       	\left\{
       	\begin{aligned}
       		& \d \omega =\nu B^{-1}\Delta AB\omega + \sqrt{2\nu}(B(-\Delta)B^*)^{\frac{1}{2}}\,\d W,\\
       		& \omega(0)=\omega_0.
       	\end{aligned}
       	\right.
       \end{equation*}
       \end{enumerate}
	\end{remark}
	
	\section{Delayed blow up for NS equations}\label{NS-I}
	In this section, we prove that some pseudo-transport noise can delay blow-up in the NS equations. In contrast to transport noise \cite{FlaLuo21,Agr26,Lan24}, for which delay of blow-up requires sufficiently large noise intensity, our main result Theorem \ref{Delay blow up-NS} holds without any restriction on the noise intensity.
	
	Consider the stochastic Navier-Stokes equation on $\Tb^3$ perturbed by the pseudo-transport noise:
	\begin{equation}\label{NS-u}
		\left\{
		\begin{aligned}
			& \d u + \Pi (u\cdot\nabla u)\,\d t =\sum_k \Lc_k u \circ \d W^k+\Delta u\,\d t,\\
			& u(0)=u_0\in H^{2b+a},
		\end{aligned}
		\right.
	\end{equation}
	where $\circ$ denotes the Stratonovich integral, $\Lc_ku= (-\Delta)^{-b}\Pi\big( \sigma_k\cdot\nabla (-\Delta)^{a+b} u\big)$, $\Pi$ is the Leray projection,  $\{\sigma_k\}$ is a finite family of smooth divergence-free vector field, and $a,b \in \R$ are parameters satisfying $2b+a>\frac{1}{2}$. 
	\begin{definition}\label{NS-def sol}
		Given a filtered probability space $(\Omega,\Fc,(\Fc_t),\Pb)$ and a sequence of independent $(\Fc_t)$-Brownian motion $\{W^k\}$, we say an $(\Fc_t)$-progressively measurable process $u$ is a strong solution to \eqref{NS-u} if the following holds:
		\begin{itemize}
			\item[(1).] $u\in C([0,T]; H^{2b+a})\cap L^2(0,T; H^{2b+a+1}),\:\Pb$-a.s.
			\item[(2).] For every divergence-free vector field $\varphi\in C^\infty (\Tb^3;\R^3)$, $\Pb$-a.s. the following identity holds:
			\begin{equation*}
				\begin{split}
					\<u(t),\varphi\> &=\<u_0,\varphi\>-\int_{0}^t\<u\cdot\nabla u, \varphi\>\,\d s + \<u,\Delta \varphi\>\,\d t\\
					&\quad- \sum_k\int_{0}^t\big\<u, (-\Delta)^{2b+a}\Lc_k (-\Delta)^{-2b-a}\varphi\big\> \circ \d W_s^k,\quad \text{ for all } t\in [0,T].
				\end{split}
			\end{equation*}
		\end{itemize}
	\end{definition}
	\begin{remark}
		\leavevmode
		\begin{itemize}
			\item Due to the high regularity of the strong solution, the Stratonovich integral is well defined and equation \eqref{NS-u} can be written into It\^o's form:
			\begin{align*}
				\d u + \Pi (u\cdot\nabla u)\,\d t =\sum_k \Lc_k u\,\d W^k 
				+\frac{1}{2}\sum_k \Lc_k^2 u\,\d t,
			\end{align*}
			where $\Lc_k^2u=(-\Delta)^{-b}\Pi\big( \sigma_k\cdot\nabla (-\Delta)^{a}\Pi\big( \sigma_k\cdot\nabla (-\Delta)^{a+b} u\big)\big)$.
			\item It is easy to check that the noise preserves the $H^{2b+a}$ norm path-wise.
		\end{itemize}
	\end{remark}
	
	Now we choose $\sigma_k=\sigma_k^N$ as in \eqref{theta^N} where the noise intensity $\nu > 0$ is involved and obtain
	\begin{equation}\label{NS-u_n}
		\left\{
		\begin{aligned}
			& \d u^N + \Pi (u^N\cdot\nabla u^N)\,\d t =\sum_k \Lc_k^N u^N  \circ \d W^k +\Delta u^N\,\d t,\\
			& u^N(0)=u_0\in H^{2b+a},
		\end{aligned}
		\right.
	\end{equation}
	where $\Lc_k^N u=(-\Delta)^{-b}\Pi\big( \sigma_k^N\cdot\nabla (-\Delta)^{a+b} u\big)$. 
	The main result of this section is that pseudo-differential transport noise can delay blow-up in the Navier-Stokes equations. 
	
	\begin{theorem}\label{Delay blow up-NS}
		Let $\frac{1}{4}\leq a<\frac{1}{2}<2b+a$ and $\nu>0$. Fix $R_0>0$ and $0<\epsilon<1$, then there exists $N_0=N_0(\nu,R_0,\epsilon)$, such that for all $N>N_0$ and initial data $u_0$ satisfying 
		\begin{equation*}
			u_0\in H^{2b+a} \quad \text{and}\quad \|u_0\|_{H^{2b+a}}\leq R_0, 
		\end{equation*}
		the equation \eqref{NS-u_n} admits a unique strong solution which is spatially smooth and, with a probability no less than $1-\epsilon$, exists globally on $[0,\infty)$.
	\end{theorem}
	
	The proof of Theorem \ref{Delay blow up-NS} is quite long and relies heavily on the scaling limit results obtained in the previous section.
	\subsection{Well-posedness for stochastic NS equations with cut-off}    
	The results in this subsection remain valid for all sufficiently smooth and divergence-free vector field $\{\sigma_k\}$, so we omit the superscript $N$. 
	To obtain strong solutions global in time, we cut off the nonlinear term. For $R>0$, let $f_R: \R\rightarrow [0,1]$ be a smooth non-increasing function taking value $1$ on $[0,R]$ and $0$ on $[R,\infty)$. Fix a number $\delta\in (0,2b+a-\frac{1}{2})$, consider
	\begin{equation}\label{NS-cut off}
		\d u +  
		f_R\big(\|u\|_{H^{2b+a-\delta}}\big)\Pi(u\cdot\nabla u)\,\d t
		=\sum_k  \Lc_k u \circ \d W^k + \Delta u\,\d t.
	\end{equation}
	One can actually show that when $a<\frac{1}{2}$, \eqref{NS-cut off} admits a unique global strong solutions to. In particular, the solution is smooth in the spatial variables.
	
	\begin{proposition}\label{GWP-NS-cut-off}
		Fix $a<\frac{1}{2}<2b+a$. For every initial data $u_0\in H^{2b+a}$, there exists a unique strong solution to \eqref{NS-cut off} in the sense of Definition \ref{NS-def sol}. Moreover, the solution instantaneously gains regularity in time and space 
		\begin{equation*}
			u\in C_t^{1/2-}C_x^{\infty},\;\Pb\text{-a.s. for every } t>0,\; x\in \Tb^3.
		\end{equation*} 
		In particular, It\^o's formula can be applied to $\|u\|_{H^{2b+a}}^2$ directly, yielding
		\begin{equation*}
			\sup_{t\in [0,T]}\|u(t)\|_{H^{2b+a}}^2 + \int_{0}^{T} \|u(t)\|_{H^{2b+a+1}}^2\,\d t\leq C\big(\|u_0\|_{H^{2b+a}},R,\delta,T\big),\quad \Pb\text{-a.s.}
		\end{equation*}
		where the upper bound is independent of the choice of $\{\sigma_k\}$.
	\end{proposition}
	\begin{proof}
		We start from the a priori estimates. Let $\Lambda=(-\Delta)^{1/2}$ and for $m\in \R$, $\<\cdot,\cdot\>_{H^m}$ denote the inner product on $H^s$. That is, $\<f,g\>_{H^m}=\<\Lambda^mf, \Lambda^mg\>$. we note that $\Lc_k$ is is anti-symmetric with respect to the $H^{2b+a}$ inner product, i.e., $\<\Lc_kf,g\>_{H^{2b+a}}=-\<f,\Lc_kg\>_{H^{2b+a}}$. For $m\geq 0$, applying $\Lambda^{m}$ to \eqref{NS-cut off}, we get 
		\begin{align*}
			&\d \Lambda^mu +  
			f_R\big(\|u\|_{H^{2b+a-\delta}}\big)\Pi\Lambda^m(u\cdot\nabla u)\,\d t
			\\
			&=\sum_k  \Lambda^m\Lc_k u \, \d W^k +\frac{1}{2}\sum_k  \Lambda^m\Lc_k^2 u\,\d t + \Delta \Lambda^mu\,\d t.
		\end{align*}
		Let $n=2b+a$, by It\^o's formula, we formally get
		\begin{equation}\label{Ito}
			\begin{split}
				\d \|u\|_{H^{n+m}}^2 &= -2\|u\|_{H^{n+m+1}}^2\,\d t -2 	f_R\big(\|u\|_{H^{n-\delta}}\big)\<\Lambda^m(u\cdot\nabla u), \Lambda^mu\>_{H^n}\,\d t\\ &\quad+2\sum_{k}\big\<\Lambda^m\Lc_ku,\Lambda^m u\big\>_{H^n}\,\d W^k \\
				&\quad+\sum_{k} \Big(\big\<\Lambda^m\Lc_k^2u,\Lambda^m u\big\>_{H^n} + \big\<\Lambda^m\Lc_ku,\Lambda^m\Lc_k u\big\>_{H^n}\Big)\,\d t\\
				&=: -2\|u\|_{H^{n+m+1}}^2\,\d t-2 I\,\d t+ 2\d M +\sum_kJ_k\,\d t.
			\end{split}
		\end{equation}
		We now perform a term-by-term analysis.
		\begin{itemize}
			\item Nonlinear term $I$.  Since we have cut off the nonlinearity in a subcritical space, it is classical to control it. More specifically, by  the cancellation $\<u\cdot\nabla\Lambda^{m+n}u, \Lambda^{m+n}u\>=0$, we have
			\begin{align*}
				|\<\Lambda^{m+n}(u\cdot\nabla u), \Lambda^{m+n}u\>|& = \big|\big\<[\Lambda^{m+n},u\cdot\nabla]u,\Lambda^{m+n}u\big\>\big| \\
				&\leq \|[\Lambda^{m+n},u\cdot\nabla]u\|_{L^{3/2}}\|\Lambda^{m+n}u\|_{L^3}.
			\end{align*} 
			By the Kato-Ponce commutator estimate \cite{KatPon88}, we have 
			\begin{align*}
				\|[\Lambda^{m+n},u\cdot\nabla]u\|_{L^{3/2}} \lesssim \|\Lambda^{m+n}u\|_{L^3} \|\nabla u\|_{L^3}.
			\end{align*}
			By Sobolev embedding $H^{1/2}(\Tb^3)\hookrightarrow L^3(\Tb^3)$, we arrive at 
			\begin{equation*}
				|I|\lesssim  f_R\big(\|u\|_{H^{n-\delta}}\big) \|u\|_{H^{3/2}} \|u\|_{H^{m+n+1/2}}^2.
			\end{equation*}
			By the interpolation between $H^{n-\delta}$ and $H^{m+n+1}$, we get 
			\begin{equation*}
				\|u\|_{H^{m+n+1/2}}^2 \leq \|u\|_{H^{n-\delta}}^{\frac{1}{1+\delta+m}} \|u\|_{H^{m+n+1}}^{\frac{1+2\delta+2m}{1+\delta+m}} 
			\end{equation*}
			With out loss of generality, we assume $n-\delta<3/2<m+n+1$, then still by the interpolation  between $H^{n-\delta}$ and $H^{m+n+1}$, we have
			\begin{equation*}
				\|u\|_{H^{3/2}}\leq \|u\|_{H^{n-\delta}}^{\frac{-1/2+m+n}{1+\delta+m}} \|u\|_{H^{m+n+1}}^{\frac{3/2-(n-\delta)}{1+\delta+m}}.
			\end{equation*}
			Combining these inequalities gives 
			\begin{equation*}
				|I|\lesssim f_R\big(\|u\|_{H^{n-\delta}}\big) \|u\|_{H^{n-\delta}}^{\frac{1/2+m+n}{1+\delta+m}} \|u\|_{H^{m+n+1}}^{2-\frac{n-\delta-1/2}{1+\delta+m}}.
			\end{equation*}
			Noting that $2-\frac{n-\delta-1/2}{1+\delta+m}<2$ by assumption $n-\delta>1/2$, then Young's inequality yields 
			\begin{equation}\label{I-bdd}
				|I|\leq \frac{1}{4} \|u\|_{H^{m+n+1}}^2 + C(R,\delta,m,n)
			\end{equation}
			\item Noise induced term $J_k$.  We argue using commutator estimates similar to these in \cite[Section 4.4]{CFH19}. Let $S_k$ denote the commutator $[\Lambda^m, \Lc_k]$ and $T_k$ denote the commutator $[S_k, \Lc_k]$. By classical theory of pseudo-differential operators (c.f. \cite{Hor07}), $S_k$ (resp. $T_k$) is of order $m+2a$ (resp. $m+4a$). Recalling that $\Lc^k$ is anti-symmetric on $H^n$, we have
			\begin{align*}
				\<\Lambda^m\Lc_k^2u,\Lambda^m u\big\>_{H^n} &= \< S_k\Lc_ku,\Lambda^m u\big\>_{H^n} + \<\Lc_k\Lambda^m\Lc_k u,\Lambda^m u\big\>_{H^n}\\
				&=\< S_k\Lc_ku,\Lambda^m u\big\>_{H^n} - \<\Lambda^m\Lc_k u,\Lc_k\Lambda^m u\big\>_{H^n}.
			\end{align*}
			Hence we have 
			\begin{align*}
				J_k &= \< S_k\Lc_ku,\Lambda^m u\big\>_{H^n} + \<\Lambda^m\Lc_k u,S_k u\big\>_{H^n} \\
				&=\< S_k\Lc_ku,\Lambda^m u\big\>_{H^n} + \<S_k u,S_k u\big\>_{H^n} + \<\Lc_k\Lambda^m u,S_k u\big\>_{H^n} \\
				&= \< S_k\Lc_ku,\Lambda^m u\big\>_{H^n} + \<S_k u,S_k u\big\>_{H^n} - \<\Lambda^m u,\Lc_kS_k u\big\>_{H^n} \\
				&= \<S_k u,S_k u\big\>_{H^n} + \< T_ku,\Lambda^m u\big\>_{H^n},
			\end{align*}
			which implies
			\begin{equation*}
				|J_k|\leq \|S_k u\|_{H^n}^2+\|T_k u\|_{H^{n-2a}}\|\Lambda^m u\|_{H^{n+2a}}\lesssim_k \|u\|_{H^{m+n+2a}}^2.
			\end{equation*}
			Since $a<\frac{1}{2}$ by assumption, it follows that 
			\begin{equation}\label{J_k-bdd}
				|J_k|\leq \frac{1}{2}\|u\|_{H^{m+n+1}}^2 +C_k\|u\|_{H^{m+n}}^2.
			\end{equation}
			\item Martingale term $M$. For the quadratic variation process $\<M\>$, the following holds:
			\begin{align*}
				\<M\>_T &=\sum_k\int_{0}^T \<\Lambda^m\Lc_k u,\Lambda^m u\>_{H^n}^2 \,\d t\\
				&=\sum_k\int_{0}^T \<S_k u,\Lambda^m u\>_{H^n}^2 \,\d t\lesssim_{\{\sigma_k\}}\int_{0}^T \|u\|_{H^{m+n+a}}^4\,\d t.
			\end{align*} 
			Hence by interpolation inequality and H\"older's inequality, we have
			\begin{equation}\label{M-bdd}
				\begin{split}
					\<M\>_T^{1/2}&\lesssim \bigg(\int_{0}^T\|u\|_{H^{m+n}}^2\|u\|_{H^{m+n+2a}}^2\,\d t\bigg)^{1/2}\\
					&\lesssim\sup_{t\in [0,T]}\|u\|_{H^{m+n}}\bigg(\int_{0}^T\|u\|_{H^{m+n+2a}}^2\,\d t\bigg)^{1/2} \\
					&\leq \epsilon \sup_{t\in [0,T]}\|u\|_{H^{m+n}}^2 + C_\epsilon \int_{0}^T\|u\|_{H^{m+n+2a}}^2\,\d t,
				\end{split}
			\end{equation}
			where $\epsilon>0$ is a constant to be determined later.
		\end{itemize}
		Now substituting the estimates \eqref{I-bdd}-\eqref{J_k-bdd} into \eqref{Ito}, integrating on time and taking expectation, by Gr\"onwall's inequality, we obtain
		\begin{equation*}
			\sup_{t\in [0,T]} \E\Big[\|u\|_{H^{m+n}}^2\Big] + \int_{0}^T \E\Big[\|u\|_{H^{m+n+1}}^2\Big]\,\d t\leq C\big(\|u_0\|_{H^{m+n}},R,\delta,m,n,\{\sigma_k\}\big).
		\end{equation*}
		To move the supremum over time inside the expectation, we exploit the Burkholder-Davis-Gundy inequality and \eqref{M-bdd}, and obtain
		\begin{align*}
			\E\bigg[\sup_{t\in [0,T]}\|u\|_{H^{m+n}}^2\bigg] &\lesssim \|u_0\|_{H^{m+n}}^2 + \int_{0}^T \E\Big[\|u\|_{H^{m+n}}^2\Big]\,\d t + \E\bigg[\sup_{t\in [0,T]}|M_t|\bigg] \\
			&\lesssim\|u_0\|_{H^{m+n}}^2 + \int_{0}^T \E\Big[\|u\|_{H^{m+n}}^2\Big]\,\d t + \E[\<M_T\>^{1/2}]\\
			&\lesssim \|u_0\|_{H^{m+n}}^2  +\epsilon \E\bigg[\sup_{t\in [0,T]}\|u\|_{H^{m+n}}^2\bigg] + C_\epsilon \int_{0}^T\E\bigg[\|u\|_{H^{m+n+2a}}^2\bigg]\,\d t.
		\end{align*}
		Choosing $\epsilon>0$ sufficiently small, we arrive at
		\begin{equation}\label{apriori-bdd}
			\E\bigg[\|u\|_{C([0,T];H^{m+n})}^2\bigg] \leq C\big(\|u_0\|_{H^{m+n}},R,\delta,m,n,\{\sigma_k\}\big).
		\end{equation}
		It is worth noting that when $m=0$, the noise terms $J_k$ and $M$ vanish, hence the underlying constant in \eqref{apriori-bdd} is independent of $\{\sigma_k\}$, and we obtain a pathwise bound.
		
		Next, we derive a rigorous It\^o's formula for $\|u\|_{H^{n}}^2=\|u\|_{H^{2b+a}}^2$.The case $a \leq 0$ is a standard consequence of the classical infinite-dimensional It\^o formula \cite[Theorem 4.2.5]{LR15}. In contrast, the case $0<a<1/2$ demands some extra work. Denote by $J_\epsilon$ the standard mollification operator on $\Tb^3$. By replacing the operator $\Lambda^m$ by $J_\epsilon$ in the preceding argument, we have $\Pb\text{-a.s.}$, for every $t\in [0,T]$,
		\begin{equation*}
			\begin{split}
				\|J_\epsilon u\|_{H^n}^2 &= \|J_\epsilon u_0\|_{H^n}^2 -2 \int_{0}^t\|J_\epsilon u\|_{H^{n+1}}^2\,\d s-2\sum_k\int_{0}^t\<S_k^\epsilon u, J_\epsilon u\>_{H^n}\d W_s^k\\ &\quad-2\int_{0}^tf_R\big(\|u\|_{H^{n-\delta}}\big)\big\<[J_\epsilon\Lambda^n,u\cdot\nabla] u, J_\epsilon\Lambda^n u\big\>\,\d s \\
				&\quad +\sum_k\int_{0}^t\<S_k^\epsilon u,S_k^\epsilon u\big\>_{H^n} + \< T_k^\epsilon u,J_\epsilon u\big\>_{H^n}\,\d s  ,
			\end{split}
		\end{equation*}
		where $S_k^\epsilon= [J_\epsilon, \Lc_k]$ and $T_k^\epsilon= [S_k^\epsilon, \Lc_k]$. We claim that, as $\epsilon\rightarrow 0$,
		\begin{align}
			&\big\<[J_\epsilon\Lambda^n,u\cdot\nabla] u, J_\epsilon\Lambda^n u\big\>\rightarrow \big\<[\Lambda^n,u\cdot\nabla] u, \Lambda^n u\big\>, \label{commutator1}\\
			&\|S_k^\epsilon\|_{H^n}\lesssim_k \epsilon^{1-2a} \|u\|_{H^{n+1}},\quad \|T_k^\epsilon\|_{H^{n-2a}}\lesssim_k  \epsilon^{1-2a} \|u\|_{H^{n+1}}.\label{commutator2}
		\end{align}
		With the above claim in hand, we let $\epsilon\rightarrow 0$ to obtain $\Pb\text{-a.s.},$ for every $t\in [0,T]$,
		\begin{equation*}
			\begin{split}
				\| u\|_{H^n}^2 = \|  u_0\|_{H^n}^2 -2 \int_{0}^t\| u\|_{H^{n+1}}^2\,\d s -2\int_{0}^tf_R\big(\|u\|_{H^{n-\delta}}\big)\big\<[\Lambda^n,u\cdot\nabla] u, \Lambda^n u\big\>\,\d s.
			\end{split}
		\end{equation*}
		The proof of the claims \eqref{commutator1}-\eqref{commutator2} are deferred to Appendix \ref{proof of commutator}.
		
		With these preparations in hand, the proof of Proposition \ref{GWP-NS-cut-off} becomes straightforward. Indeed, existence of solutions follows from a standard Galerkin approximation. For the pathwise uniqueness, we apply It\^o's  formula to $\|u^1-u^2\|_{H^{2b+a}}^2$, where $u^1$ and $u^2$ are two solutions starting from the same initial data $u_0\in H^{2b+a}$. Since all terms induced by noise vanish, the proof proceeds exactly as in the deterministic case, and is thus omitted here. Finally, the instantaneous smoothing effect is a direct consequence of a standard bootstrap argument.
	\end{proof}
	
	Analogous to the argument in \cite[Section 3.2]{CFH19}, Proposition \ref{GWP-NS-cut-off} implies the local well-posedness of the untruncated equation \eqref{NS-u}. Now we prove that small initial data implies global well-posedness.
	\begin{proposition}\label{small-GWP}
		Fix $a<\frac{1}{2}<2b+a$. Then there exists a constant $\delta_0=\delta_0(a,b)$, independent of the choices of $\{\sigma_k\}$, such that for every initial data 
		\begin{equation*}
			u_0\in H^{2b+a}, \quad \|u_0\|_{H^{2b+a}}\leq \delta_0,
		\end{equation*}
		the equation \eqref{NS-u} admits a unique global strong solution.
	\end{proposition}
	
	\begin{proof}
		We only need to show the solution is global in time. Let $n=2b+a>1/2$. Proceeding with the same energy estimates as in the proof of Proposition \ref{GWP-NS-cut-off}, we obtain 
		\begin{equation*}
			\frac{\d}{\d t} \|u\|_{H^{n}}^2 +2\|u\|_{H^{n+1}}^2\lesssim_{n} \|u\|_{H^{3/2}}\|u\|_{H^{n}}\|u\|_{H^{n+1}},\quad \Pb\text{-a.s.}
		\end{equation*}
		If $n>3/2$, we simply use $\|u\|_{H^{3/2}}\leq \|u\|_{H^n}$. Otherwise, by interpolation between $H^n$ and $H^{n+1}$, we have
		\begin{equation*}
			\|u\|_{H^{3/2}}\leq \|u\|_{H^n}^{n+1/2}\|u\|_{H^{n+1}}^{5/2-n}.
		\end{equation*}
		In both cases, by H\"older's inequality, we get $\Pb\text{-a.s.}$,
		\begin{equation*}
			\frac{\d}{\d t} \|u\|_{H^{n}}^2 +\|u\|_{H^{n+1}}^2\lesssim_{n} \|u\|_{H^{n}}^{p(n)}, 
		\end{equation*}
		where $p(n)=\frac{2(2n+1)}{2n-1}$ for $n\in (1/2,3/2)$ and $p(n)=4$ for $n\geq 3/2$. Let $x(t)=\|u(t)\|_{H^n}^2$, we have 
		\begin{equation*}
			\dot{x} \leq -x + C(n) x^{p(n)}, \quad x(0)\leq \delta_0^2.
		\end{equation*}
		Obviously, provided that $\delta_0^2<C(n)^{-\frac{1}{p(n)-1}}$, the function $x(t)$ decreases in time; hence, the solution exists globally.
	\end{proof}
	\subsection{Scaling limit} 
	
	By Proposition \ref{GWP-NS-cut-off}, for each $R,T>0$ and $N\geq 1$, the following equation
	\begin{equation}\label{NS-u_n-cutoff}
		\left\{
		\begin{aligned}
			&\d u^N +  
			f_R\big(\|u\|_{H^{2b+a-\delta}}^N\big)\Pi(u^N\cdot\nabla u^N)\,\d t
			=\sum_k  \Lc_k^N u^N \circ \d W^k + \Delta u^N\,\d t,\\
			&u^N(0)=u_0\in H^{2b+a}\quad\text{and}\quad \|u_0\|_{H^{2b+a}}\leq R_0 
		\end{aligned}
		\right.
	\end{equation}
	admits a unique strong solution $u^N\in C([0,T]; H^{2b+a})\cap L^2(0,T; H^{2b+a+1}),\:\Pb$-a.s., and we have uniform estimate
	\begin{equation}\label{bound1}
		\sup_{N} \Big(\|u^N\|_{C([0,T];H^{2b+a})} + \|u^N\|_{L^2(0,T;H^{2b+a+1})}\Big) \leq C(R_0,R,
		\delta,T),\quad \Pb\text{-a.s.}
	\end{equation} 
	We remark on a slight abuse of notation here: $u^N$ is used to denote the solution to the truncated equation \eqref{NS-u_n-cutoff}, while it also refers to the solution of the original equation \eqref{NS-u_n}. As shown in the following proof, by carefully choosing the truncation parameter $R$, the solution of \eqref{NS-u_n-cutoff} coincides with that of \eqref{NS-u_n} with high probability. 
	\begin{theorem}\label{scaling limit-cut-off}
		Fix  $\nu,R_0,R,T>0$ and $ a<\frac{1}{2}<2b+a-\delta$, then $u^N$ converges in probability, in space $\mathcal{X}=L^2(0,T;H^{2b+a})\cap C([0,T];H^{2b+a-\delta})$ to the unique global strong solutions to the deterministic equation with hyperviscosity 
		\begin{equation}\label{hypervis-u-cutoff}
			\left\{
			\begin{aligned}
				& \partial_t u + f_R(\|u\|_{H^{2b+a-\delta}})\Pi (u\cdot\nabla u) =-\frac{3\nu}{5}(-\Delta)^{1+a}u+\Delta u,\\
				& u(0)=u_0.
			\end{aligned}
			\right.
		\end{equation}
	\end{theorem}
	\begin{remark}
		It is well known that when $a\geq1/4$, the limit equation is global well-posedness even without the cut-off on nonlinearity. See Proposition \ref{critical-hypervis} for more details.
	\end{remark}
	To apply the compactness method to prove the convergence stated in Theorem \ref{scaling limit-cut-off}, we need the following estimates.
	\begin{lemma}\label{bound2}
		Let $a,b\in \R$ and $M=2b+a-2(1+a+|a|)$. For every progressively measurable process $\varphi\in L^\infty(0,T;H^{2b+a})$, define 
		\begin{align*}
			S^N (\varphi)&=\frac{1}{2}\sum_k (\Lc_k^N)^2 \varphi=(-\Delta)^{-b}\Pi\big( \sigma_k^N\cdot\nabla (-\Delta)^{a}\Pi\big( \sigma_k^N\cdot\nabla (-\Delta)^{a+b} \varphi\big)\big),\\
			M_t^N &=\sum_k\int_{0}^{t} \Lc_k^N \varphi\,\d W^k=\sum_k\int_{0}^{t}(-\Delta)^{-b}\Pi\big( \sigma_k^N\cdot\nabla (-\Delta)^{a+b}\varphi\big)\,\d W^k.
		\end{align*}
		Then for every $t,s\in [0,T]$ and $1\leq p<\infty$, we have 
		\begin{align*}
			\|S^N (\varphi)\|_{H^M}&\lesssim_{a,\nu} \|\varphi\|_{H^{2b+a}},\quad \Pb\text{-a.s.}\\ \E\Big[\|M_t^N-M_s^N\|_{H^M}^p\Big]&\lesssim_{a,\nu,p} |t-s|^{p/2} \|\varphi\|_{L^\infty (0,T;H^{2b+a})}^p.
		\end{align*}
	\end{lemma}
	\begin{proof}
		Without loss of generality, we can assume $b=0$.
		Arguing as in Lemma \ref{Ito corrections limit}, we get, for every $\xi\in \Z_0^3$,
		\begin{align*}
			2\widehat{S_N(\varphi)}(\xi)&=-(2\pi)^{-3/2}P_{\xi}^\perp\sum_{\eta} \xi^T\widehat{Q^N}(\eta)\xi\, |\xi-\eta|^{2a}P_{\xi-\eta}^\perp\: |\xi|^{2a}\hat{\varphi}(\xi)\\
			&=-(2\pi)^{-3/2}P_{\xi}^\perp\sum_{\eta} \xi^T\widehat{Q^N}(\eta)|\eta|^{2a}\xi\, \frac{|\xi-\eta|^{2a}}{|\eta|^{2a}}P_{\xi-\eta}^\perp\: |\xi|^{2a}\hat{\varphi}(\xi).
		\end{align*}
		Recall that  we have $\sup\limits_{\eta}\Big|\frac{|\xi-\eta|^{2a}}{|\eta|^{2a}}\Big|\lesssim_a |\xi|^{2|a|}$ and $\sum_{\eta}|\widehat{Q^N}(\eta)|\,|\eta|^{2a}\lesssim \nu$ and by \eqref{intermidiate}, we arrive at
		\begin{equation*}
			\big|\widehat{S^N(\varphi)}(\xi)\big|\lesssim_{a,\nu} |\xi|^{2(1+a+|a|)} |\hat{\varphi}(\xi)|,
		\end{equation*}
		which implies the first assertion. For the second one, by Burkholder-Davis-Gundy inequality, we have
		\begin{equation}\label{BDG}
			\E\Big[\|M_t^N-M_s^N\|_{H^M}^p\Big]\lesssim_p \E\Big[\<M^N\>_{s,t}^{p/2}\Big],
		\end{equation}
		where the quadratic variation $\<M^N\>_{s,t}=\sum_k\int_{s}^{t}\|\Lc_k^N\varphi(r)\|_{H^M}^2$\,\d r. Note that 
		\begin{align*}
			\sum_k\|\Lc_k^N\varphi\|_{H^M}^2 &= \big\<\Pi\big( \sigma_k^N\cdot\nabla (-\Delta)^{a}\varphi\big), (-\Delta)^M\big( \sigma_k^N\cdot\nabla (-\Delta)^{a}\varphi\big)\big\> \\
			&\leq (2\pi)^{-\frac{3}{2}}\sum_{\xi,\eta} |\xi|^{4a}|\hat{\varphi}(\xi)|^2|\widehat{Q^N}(\eta)|\,|\eta|^{2a} \frac{|\xi-\eta|^{2M+2}}{|\eta|^{2a}} 
		\end{align*}
		A case by case analysis for $|\eta|\leq |\xi|/2,\,|\xi|/2\leq|\eta|\leq 2|\xi|$ and $|\eta|\geq 2|\xi|$ yields the estimate 
		\begin{equation*}
			\sup_{\eta}\frac{|\xi-\eta|^{2M+2}}{|\eta|^{2a}}\lesssim_{a} |\xi|^{-2a}.
		\end{equation*}
		Hence we obtain
		\begin{equation*}
			\sum_k\|\Lc_k^N\varphi\|_{H^M}^2\lesssim_{a}\sum_{\xi} |\xi|^{2a}|\hat{\varphi}(\xi)|^2\sum_{\eta}|\widehat{Q^N}(\eta)|\,|\eta|^{2a} \lesssim_{a,\nu}  \|\varphi\|_{H^a}^2.
		\end{equation*}
		Substituting into \eqref{BDG}, we get 
		\begin{equation*}
			\E\Big[\|M_t^N-M_s^N\|_{H^M}^p\Big]\lesssim_{a,\nu,p} |t-s|^{p/2}\|\varphi\|_{L^\infty (0,T;H^{a})}^p,
		\end{equation*}
		which completes the proof.
	\end{proof}
	\begin{lemma}\label{S^N limit}
		For every divergence-free vector field $\varphi\in C^\infty(\Tb^3;\R^3)$,  the following limit holds in $H^r(\Tb^3;\R^3)$ for every $r\in\R$:
		\begin{equation*}
			\lim\limits_{N\rightarrow\infty}S^N(\varphi)=-\frac{3\nu}{5}(-\Delta)^{1+a}\varphi.
		\end{equation*}
	\end{lemma}
	\begin{proof}
		It suffices to prove the convergence of the corresponding Fourier transform. Note that for every $\xi\in\Z_0^3$,
		\begin{equation}\label{Fourier-S^N}
			2\widehat{S^N(\varphi)}(\xi)
			=-(2\pi)^{-3/2}P_{\xi}^\perp\sum_{\eta} \xi^T\widehat{Q^N}(\eta)|\eta|^{2a}\xi\, \frac{|\xi-\eta|^{2a}}{|\eta|^{2a}}P_{\xi-\eta}^\perp\: |\xi|^{2a}\hat{\varphi}(\xi).
		\end{equation}
		Define operator $T^N$ as 
		\begin{equation*}
			2\widehat{T^N(\varphi)}(\xi) =-(2\pi)^{-3/2}P_{\xi}^\perp\sum_{\eta}\xi^T\widehat{Q^N}(\eta)|\eta|^{2a}\xi\, P_{\xi-\eta}^\perp\: |\xi|^{2a}\hat{\varphi}(\xi)
		\end{equation*}
		By \cite[Appendix B]{LuoTan23}, 	$2\widehat{T^N(\varphi)}(\xi)\rightarrow -\frac{6\nu}{5}\widehat{(-\Delta)^{1+a}\varphi}(\xi)$, as $N\rightarrow \infty$. Additionally, we note that
		\begin{equation*}
			\big|\widehat{S^N(\varphi)}-\widehat{T^N(\varphi)}\big|(\xi)\lesssim |\xi|^{2a+2}|\hat{\varphi}(\xi)|\sum_\eta |\widehat{Q^N}(\eta)|\,|\eta|^{2a} \bigg|\frac{|\xi-\eta|^{2a}}{|\eta|^{2a}}-1 \bigg|,
		\end{equation*}
		which converges to $0$ by the proof of Lemma \ref{Ito corrections limit}, and the proof is complete.
	\end{proof}
	
	Based on Lemma \ref{bound2}-\ref{S^N limit}, we are now in a position to prove Theorem \ref{scaling limit-cut-off}.
	\begin{proof}[Proof of Theorem \ref{scaling limit-cut-off}]
		We only give a sketch of proof since it is similar to that of \cite[Theorem 1.4]{FlaLuo21}. First,write \eqref{NS-cut off} in integral form:
		\begin{align*}
			u^N(t)-u^N(s) &= -f_R(\|u^N\|_{H^{2b+a-\delta}})\int_{s}^{t} \Pi(u^N\cdot\nabla u^N)\,\d r\\
			&\quad + \sum_k\int_{s}^{t} \Lc_k^N u^N\d W_r^k + \int_{s}^{t} S^N(u^N)\,\d r.
		\end{align*}
		Let $M=2b+a-2(1+a+|a|)$, then by Lemma \ref{bound2} and \eqref{bound1}, for every $p<\infty$, we have
		\begin{align*}
			\E\Big[\|u^N(t)-u^N(s)\|_{H^M}^p\Big] &\lesssim_{a,\nu,p} C(R_0,R,\delta,T)^p (|t-s|^{p/2}+|t-s|^p) \\
			&\quad + \E\bigg[\Big\|\int_{s}^{t} \Pi(u^N\cdot\nabla u^N)\,\d r\Big\|_{H^M}^p\bigg].
		\end{align*}
		For the nonlinear term, since $2b+a>1/2$, it is easy to check that 
		\begin{equation*}
			\|u^N\cdot\nabla u^N\|_{H^M}= \|u^N\otimes u^N\|_{H^{M+1}}\lesssim_{a,b} \|u^N\|_{H^{2b+a}}^2.
		\end{equation*}
		Then we have
		\begin{align*}
			\E\bigg[\Big\|\int_{s}^{t} \Pi(u^N\cdot\nabla u^N)\,\d r\Big\|_{H^M}^p\bigg]
			&\leq 	\E\bigg[\Big(\int_{s}^{t} \big\|\Pi(u^N\cdot\nabla u^N)\big\|_{H^M}\,\d r\Big)^p\bigg] \\
			&\lesssim_{a,b} C(R_0,R,\delta,T)^{2p} |t-s|^p.
		\end{align*}
		Thus we obtain the estimate 
		\begin{equation*}
			\E\Big[\|u^N(t)-u^N(s)\|_{H^M}^p\Big]\lesssim |t-s|^{p/2},
		\end{equation*}
		where the underlying constant is independent of $N$. The above uniform bound implies that for all $0<\beta<1/2$ and $1\leq p<\infty$, we have
		\begin{equation}\label{time-C^0.5}
			\sup_{N}\E\Big[\|u^N\|_{W^{\beta,p}(0,T;H^M)}^p\Big]\leq C(R_0,R,\delta,T,p,a,b,\beta)<\infty.
		\end{equation}
		Estimates \eqref{bound1} and \eqref{time-C^0.5} directly imply that the law of $u^N$ is tight on  space $\mathcal{X}=L^2(0,T;H^{2b+a})\cap C([0,T];H^{2b+a-\delta})$. Following the same compactness argument as in \cite[Section 4]{FlaLuo21}, it remains to prove that the limit $u$ solves equation \eqref{hypervis-u-cutoff}. For every divergence-free vector field $\varphi\in C^{\infty}$, we have $\Pb\text{-a.s.}$,
		\begin{equation*}
			\begin{split}
				\<u^N(t),\varphi\> &=\<u_0,\varphi\>-\int_{0}^tf_R(\|u^N\|_{H^{2b+a-\delta}})\<u^N\cdot\nabla u^N, \varphi\>\,\d s + \<u^N,\Delta \varphi\>\,\d t\\
				&\quad- \sum_k\int_{0}^t\big\<\sigma_k^N\cdot\nabla (-\Delta)^{a+b} u^N,  (-\Delta)^{b}\varphi\big\>\, \d W_s^k\\
				&\quad+\int_{0}^{t}\big\<u, (-\Delta)^{2b+a}S^N (-\Delta)^{-2b-a}\varphi \big\>\,\d t,\quad \text{ for all } t\in [0,T].
			\end{split}
		\end{equation*}
		We can pass the limit for the nonlinear term by convergence in $C([0,T];H^{2b+a-\delta})$, the martingale vanishes in the limit by Lemma \ref{det-martingale-limit} and the convergence for the It\^o-Stratonovich term follows from Lemma \ref{S^N limit}.
	\end{proof}
	
	\subsection{Proof of Theorem \ref{Delay blow up-NS}}
	\begin{proof}
		\textbf{Part 1.}
		Assume $\frac{1}{4}\leq a< \frac{1}{2}< 2b+a-\delta$  . Fix $R_0,\nu,T>0$ and $\|u_0\|_{H^{2b+a}}\leq R_0$. By the uniform bound \eqref{s} in the proof of Proposition \ref{critical-hypervis}, we can choose $R>R_0e^{\frac{C_{a,b}}{\nu^2}R_0^2}+1$, then the limit $u$ satisfies 
		\begin{equation*}
			\partial_t u + \Pi (u\cdot\nabla u) =-\frac{3\nu}{5}(-\Delta)^{1+a}u+\Delta u.
		\end{equation*}
		Moreover, we have estimate
		\begin{equation*}
			\|u\|_{C([0,T];H^{2b+a})}\leq R_0e^{\frac{C_{a,b}}{\nu^2}R_0^2}=R-1.
		\end{equation*}
		Given $\epsilon\in (0,1)$, Theorem \ref{scaling limit-cut-off} implies that there exists $N_0=N_0(\nu,R_0,R,T,\epsilon)\in\N$, such that the unique strong solution $u^N$ to \eqref{NS-u_n-cutoff} satisfies
		\begin{equation*}
			\Pb\Big(\|u^N-u\|_{C([0,T];H^{2b+a-\delta})}>\epsilon\Big) + 	\Pb\Big(\|u^N-u\|_{L^2(0,T;H^{2b+a})}>\epsilon\Big)<\epsilon.
		\end{equation*}
		Define $\Omega_\epsilon =\Big\{\|u^N-u\|_{C([0,T];H^{2b+a-\delta})}\vee\|u^N-u\|_{L^2(0,T;H^{2b+a})}\leq \epsilon\Big\}$, then $\Pb(\Omega_\epsilon)\geq 1-\epsilon$. Clearly, the cut off function equals $1$ on $\Omega_\epsilon$, and thus $u^N$ solves \eqref{NS-u_n} on the set $\Omega_\epsilon\times[0,T]$. 
		
		\textbf{Part 2.}
		With the delayed blow-up on a finite time interval proved, we proceed to further extend the existence time of $u^N$. In what follows, set $\epsilon=\frac{\delta_0}{2}$, where $\delta_0$ is the constant in Proposition \ref{small-GWP}. The exponential decay of $\|u\|_{H^{2b+a}}$ \eqref{s-exp} implies that for sufficiently large $T$, $\|u(t)\|_{H^{2b+a}}\leq \epsilon$ holds for every $t\in [T-1,T]$. We note that on $\Omega_\epsilon$, we have
		\begin{equation*}
			\|u^N\|_{L^2(0,T;H^{2b+a})}\leq \|u^N\|_{L^2(0,T;H^{2b+a})} + \|u^N\|_{L^2(0,T;H^{2b+a})} \leq \epsilon+\epsilon =\delta_0
		\end{equation*}
		For every $\omega\in \Omega_\epsilon$, there exists $t=t(\omega)\in [T-1,T]$ such that
		\begin{equation*}
			\|u^N(t(\omega),\omega)\|_{H^{2b+a}}\leq \delta_0.
		\end{equation*}
		By Proposition \ref{small-GWP}, with such small initial data, the solution can be extended globally in time, which proves the desired result.
	\end{proof}
	
	\begin{remark}\label{NS-rmk}
		For $-1 \leq a < \frac{1}{4}$ (where $a=-1$ corresponds to the linear damping), the additional dissipation $\nu(-\Delta)^{1+a}$ does not yield global regularity. Consequently, well-posedness of the limit equation is established only under the condition of large noise intensity $\nu$ relative to the initial data (c.f. Proposition \ref{limit-NS}). This leads to the following result:
		\begin{theorem}
			Let $-1\leq a<\frac{1}{4}<\frac{1}{2}<2b+a$. Fix $R_0>0$ and $0<\epsilon<1$, then there exist $\nu_0=\nu_0(R_0)$ and $N_0=N_0(\nu,R_0,\epsilon)$, such that for all $\nu>\nu_0,\;N>N_0$ and initial data $u_0$ satisfying 
			\begin{equation*}
				u_0\in H^{2b+a} \quad \text{and}\quad \|u_0\|_{H^{2b+a}}\leq R_0, 
			\end{equation*}
			the equation \eqref{NS-u_n} admits a unique strong solution which is spatially smooth and, with a probability no less than $1-\epsilon$, exists globally on $[0,\infty)$.
		\end{theorem}
		
		As indicated by \eqref{J_k-bdd}, control over the noise is maintained only when $a < \frac{1}{2}$. In the case of $a > \frac{1}{2}$, the $H^s$ a priori estimate closes if and only if $s = 2b + a$. Under the assumption $\min\{a, 2b+a\}>\frac{1}{2}$, the weak existence of solutions (in a probabilistic sense) to the truncated equation \eqref{NS-cut off} can be shown, whereas path-wise uniqueness and the instantaneous smoothing effect cannot be established. Recall that a weak solution on $[0,T]$ refers to a filtered probability space $(\Omega, \Fc, (\Fc_t), \Pb)$ equipped with a sequence of independent $(\Fc_t)$-Brownian motions $\{W^k\}$ and an $(\Fc_t)$-progressively measurable process $u$ satisfying $(1)$--$(2)$ in Definition \ref{NS-def sol}, where $C([0,T];H^{2b+a})$ is replaced by $L^\infty([0,T];H^{2b+a})$.
		This leads to the following:
		\begin{proposition}\label{a>0.5}
			Let $\min\{a, 2b+a\}>\frac{1}{2}$ and $\nu>0$. Fix $R_0,T>0$ and $0<\epsilon<1$, then there exists $N_0=N_0(\nu,R_0,\epsilon,T)$, such that for all $N>N_0$ and initial data $u_0$ satisfying 
			\begin{equation*}
				u_0\in H^{2b+a} \quad \text{and}\quad \|u_0\|_{H^{2b+a}}\leq R_0, 
			\end{equation*}
			we can construct a weak solution (in a probabilistic sense) to \eqref{NS-u_n} up to time $T$ with a probability no less than $1-\epsilon$.
		\end{proposition}
		We point out that the strategy used in Theorem \ref{Delay blow up-NS} -- extending the solution's lifespan to infinity by exploiting the smallness of the data at time $T$ -- fails to carry over to the current setting. Although the solution remains small at $T$, the probabilistically weak solution constructed starting from $T$ may be defined on a different probability space.
		
		Finally, although we can prove path-wise uniqueness for every $a\in \R$ in the absence of the nonlinearity (c.f. Appendix \ref{linear-unique}), the method there does not readily extend to nonlinear setting, see \cite[p862, Remark]{Gal20} for details.
	\end{remark}

	\section{Delayed blow up for Euler equations}\label{Euler}
	In this section, we discuss the regularized effect by pseudo-transport noise for 3D Euler equations on torus. Consider the stochastic Euler equation on $\Tb^3$ perturbed by the pseudo-transport noise:
	\begin{equation}\label{Euler-u}
		\left\{
		\begin{aligned}
			& \d u^N + \Pi (u^N\cdot\nabla u^N)\,\d t =\sum_k \Lc_k u^N \circ \d W^k,\\
			& u^N(0)=u_0\in H^{2b+a},
		\end{aligned}
		\right.
	\end{equation}
	where $\circ$ denotes the Stratonovich integral, $\Lc_ku= (-\Delta)^{-b}\Pi\big( \sigma_k^N\cdot\nabla (-\Delta)^{a+b} u\big)$, $\Pi$ is the Leray projection,  $\{\sigma_k^N\}$ is defined as in \eqref{theta^N}, and $a,b \in \R$ are parameters satisfying $2b+a>\frac{5}{2}$. 
	The concepts of strong and weak solutions, in the probabilistic sense, are analogous to those introduced in the previous section, with two modifications: first, the regularity of the solution is now required to be in $L^\infty([0, T]; H^{2b+a})$; second, the viscosity term in Definition \ref{NS-def sol} (2) is omitted.
	\begin{theorem}\label{Delay blow-up Euler}
		Let $-1\leq a\leq 0$ and $2b+a>\frac{5}{2}$. Fix $R_0,T>0$ and $0<\epsilon<1$, then there exist $\nu_0=\nu_0(R_0)$ and $N_0=N_0(\nu,R_0,T,\epsilon)$, such that for all $\nu>\nu_0,\;N>N_0$ and initial data $u_0$ satisfying 
		\begin{equation*}
			u_0\in H^{2b+a} \quad \text{and}\quad \|u_0\|_{H^{2b+a}}\leq R_0, 
		\end{equation*}
		the equation \eqref{Euler-u} admits a unique strong solution up to time $T$ with a probability no less than $1-\epsilon$. Moreover, if $u_0\in C^\infty$, then $u$ is also smooth in space.
	\end{theorem}
	
	\begin{proof}
		The proof is completely analogous to that in Section \ref{NS-I}. Indeed, by Proposition \ref{GWP-NS-cut-off}, the noise does not introduce singularities into any energy estimate when $a\leq 0$. For the nonlinear term, with $n=2b+a>\frac{5}{2}$ and $m\geq 0$, we recall  the following classical estimate:
		\begin{equation*}
			|\<\Lambda^{m+n}(u\cdot\nabla u), \Lambda^{m+n}u\>|\lesssim_{m+n} \|\nabla u\|_{L^\infty}\|u\|_{H^{m+n}}^2\lesssim_{m+n,\delta} \|u\|_{H^{n-\delta}}\|u\|_{H^{m+n}}^2,
		\end{equation*}
		where $\delta\in (0,n-\frac{5}{2})$. Following the same arguments as in Proposition \ref{GWP-NS-cut-off}, the equation with cut off in $H^{n-\delta}$ is globally well-posed. Furthermore, the scaling limit indicates the emergence of an extra viscosity term $-\frac{3\nu}{5}(-\Delta)^{1+a}u$ in the limit equation. As shown in Proposition \ref{limit-Euler} that there exists  $\nu_0=\nu_0(R_0)$ such that the limit equation is globally well-posed provided that $\nu>\nu_0$. The desired results then follows as a direct consequence of this property and the scaling limit.
	\end{proof}
	
	Consistent with Remark \ref{NS-rmk}, for $a>0$, the noise is so singular that energy estimates can only be closed in $H^{2b+a}$. Consequently, we do not expect higher regularity even for smooth data, nor can we prove path-wise uniqueness (as estimates fail even at lower regularity). We thus only state the existence of probabilistically weak solutions:
	\begin{proposition}\label{a>0.25}
		Let $a\geq \frac{1}{4}$, $2b+a>\frac{5}{2}$ and $\nu>0$. Fix $R_0,T>0$ and $0<\epsilon<1$, then there exists $N_0=N_0(\nu,R_0,\epsilon,T)$, such that for all $N>N_0$ and initial data $u_0$ satisfying 
		\begin{equation*}
			u_0\in H^{2b+a} \quad \text{and}\quad \|u_0\|_{H^{2b+a}}\leq R_0, 
		\end{equation*}
		we can construct a weak solutions (in a probabilistic sense) to \eqref{Euler-u} up to time $T$ with a probability no less than $1-\epsilon$.
	\end{proposition}
	
	\begin{proposition}\label{a<0.25}
		Let $0<a< \frac{1}{4}$, $2b+a>\frac{5}{2}$. Fix $R_0,T>0$ and $0<\epsilon<1$, then there exists $\nu_0=\nu_0(R_0)$ and $N_0=N_0(\nu,R_0,\epsilon,T)$, such that for all $\nu>\nu_0$, $N>N_0$ and initial data $u_0$ satisfying 
		\begin{equation*}
			u_0\in H^{2b+a} \quad \text{and}\quad \|u_0\|_{H^{2b+a}}\leq R_0, 
		\end{equation*}
		we can construct a weak solutions (in a probabilistic sense) to \eqref{Euler-u} up to time $T$ with a probability no less than $1-\epsilon$.
	\end{proposition}
	
	\section{Scaling limit combined with SMR theory}\label{NS-II}
	Throughout this section, we always consider the case $a\in [-1, 0)$. In \cite{Agr26}, Agresti utilized the $L^p$ maximal regularity of the (hyper) viscous term to improve the uniform estimates \eqref{bound1} for the transport noise in scaling limit, which plays a pivotal role in proving delayed blow up noise. In this section, we demonstrate that this strategy can be perfectly adapted to our current setting.
	\subsection{Uniform SMR estimates}
	Let $S^N$ denote the It\^o-Stratonovich term as in Lemma \ref{S^N limit}. Given its structure as a Fourier multiplier, we first provide an estimate for the order of $S^N$.
	
	\begin{lemma}\label{order-S^N}
		Fix $a\in [-1,0)$. For every $\delta\in (0,1)$, there exists a constant $C_\delta$ depending only on $\delta$ and $a$, such that
		\begin{equation*}
			\big|\widehat{S^{N}}(\xi)\big|\leq \delta |\xi|^2 + C_\delta |\xi|^{2+2a}, \quad \text{ for every } \xi\in \Z_0^3. 
		\end{equation*}
	\end{lemma}
	\begin{proof}
		By \eqref{Fourier-S^N}, we have
		\begin{equation*}
			\big|\widehat{S^{N}}(\xi)\big|\lesssim |\xi|^{2+2a}\sum_{\eta\neq 0,\xi}\big|\widehat{Q^N}(\eta)\big|\,|\eta|^{2a} \frac{|\eta|^{-2a}}{|\xi-\eta|^{-2a}}.
		\end{equation*}
		For $\epsilon\in (0,1)$, the discussion is divided into three cases:
		\begin{itemize}
			\item $|\eta|\geq (1+\epsilon)|\xi|$, then $|\xi-\eta|\geq |\eta|-|\xi|\geq \frac{\epsilon}{1+\epsilon}|\eta|$, yielding
			\begin{equation*}
				\frac{|\eta|^{-2a}}{|\xi-\eta|^{-2a}}\leq \Big(\frac{1+\epsilon}{\epsilon}\Big)^{-2a}\leq 2^{-2a}\epsilon^{-2a}. 
			\end{equation*}
			\item $|\eta|\leq (1-\epsilon)|\xi|$, then $|\xi-\eta|\geq |\xi|-|\eta|\geq \epsilon|\xi|$, implying
			\begin{equation*}
				\frac{|\eta|^{-2a}}{|\xi-\eta|^{-2a}}\leq \Big(\frac{1-\epsilon}{\epsilon}\Big)^{-2a}\leq \epsilon^{-2a}. 
			\end{equation*}
			\item $|\eta|\in \Lambda_{\epsilon, \xi}:=\big((1-\epsilon)|\xi|, (1+\epsilon)|\xi|\big)$, noting that $|\xi-\eta|\geq 1$, we have
			\begin{equation*}
				\frac{|\eta|^{-2a}}{|\xi-\eta|^{-2a}}\leq 2^{-2a}|\xi|^{-2a}.
			\end{equation*}
		\end{itemize}
		Hence we arrive at
		\begin{align*}
			\big|\widehat{S^{N}}(\xi)\big|&\lesssim_a |\xi|^{2}\sum_{|\eta|\in \Lambda_{\epsilon,\xi}} \big|\widehat{Q^N}(\eta)\big|\,|\eta|^{2a} + \epsilon^{-2a}|\xi|^{2+2a}\sum_{\eta}\big|\widehat{Q^N}(\eta)\big|\,|\eta|^{2a} \\
			&:= J_1+ J_2.
		\end{align*}
		By \eqref{Q^N},  it is obvious that for every $\xi\in \Z_0^3$, we have $J_2\lesssim_\nu \epsilon^{-2a} |\xi|^{2+2a}$. For term $J_1$, we have
		\begin{align*}
			J_1\leq  \frac{|\xi|^2}{|\theta^N|_{h^a}^2}\sum_{\eta\in \Lambda_{\epsilon,\xi}} (\theta_\eta^N)^2|\eta|^{2a} 
			\lesssim  \frac{|\xi|^2}{|\theta^N|_{h^a}^2}\int_{|\eta|\in \Lambda_{\epsilon,\xi}\cap (1,N)}|\eta|^{-2\gamma}\,\d \eta,
		\end{align*}
		where $\gamma\in [0,3/2]$. We restrict our attention to the case $\gamma<3/2$, as the derivation for $\gamma = 3/2$ is entirely analogous. An elementary computation shows that
		\begin{equation*}
			J_1\lesssim \frac{|\xi|^2}{|\theta^N|_{h^a}^2}\int_{(1-\epsilon)|\xi|\wedge N}^{(1+\epsilon)|\xi|\wedge N} r^{2-2\gamma}\,\d r\lesssim_\gamma \Big((1+\epsilon)^{3-2\gamma}-\Big(\frac{1-\epsilon}{1+\epsilon}\Big)^{3-2\gamma}\Big)  |\xi|^2.
		\end{equation*}
		Combining the estimates of $J_1$ and $J_2$, we arrive at, for every $\xi\in \Z_0^3$, 
		\begin{equation*}
			\big|\widehat{S^{N}}(\xi)\big|\leq C(a,\nu,\gamma) \Big((1+\epsilon)^{3-2\gamma}-\Big(\frac{1-\epsilon}{1+\epsilon}\Big)^{3-2\gamma}\Big)  |\xi|^2 + C(a,\nu,\gamma)  \epsilon^{-2a} |\xi|^{2+2a}.
		\end{equation*}
		The assertion then follows directly by choosing $0<\epsilon<1$ such that $$C(a,\nu,\gamma) \Big((1+\epsilon)^{3-2\gamma}-\Big(\frac{1-\epsilon}{1+\epsilon}\Big)^{3-2\gamma}\Big) \leq \delta.$$
		This completes the proof.
	\end{proof}
	
	To facilitate the presentation, we introduce the notations $X_0 = H^{2b+a-1}$ and $X_1 = H^{2b+a+1}$. For any $\theta \in [0,1]$ and $p \in [1, \infty)$, we recall the following classical interpolation results:
	\begin{itemize}
		\item \textbf{Complex interpolation:} $X_\theta = (X_0, X_1)_\theta = H^{2b+a+2\theta-1}$;
		\item \textbf{Real interpolation:} $X_{\theta,p} = (X_0, X_1)_{\theta,p} = B_{2,p}^{2b+a+2\theta-1}$.
	\end{itemize}
	In particular, the trace space is denoted by $X_{tr} = X_{1-\frac{1}{p}, p}=B_{2,p}^{2b+a+1-\frac{2}{p}}$. The radonifying operators between a Hilbert space $H$ and a Banach space $X$ are denoted by $\gamma(H,X)$; for the exact definition, see the survey \cite{VNVW15}. Here we only use the canonical isomorphism from \cite[eq. (2.3)]{VNVW12}
	\begin{equation}\label{isomorphism}
		\gamma(l^2;X_{1/2}) \simeq H^{2b+a}(\Tb^3;l^2).
	\end{equation}
	
	Fix $a\in [-1,0),p\in [2,\infty)$ and progressively measurable processes $f\in L^p(\Omega\times\R_+;X_0),g=(g_k)\in L^p(\Omega\times\R_+;\gamma(l^2;X_{1/2}))$. Consider the following stochastic Stokes equations:
	\begin{equation}\label{Stokes-v}
		\left\{
		\begin{aligned}
			& \d v^N =\Delta v^N\,\d t + \sum_k (\Lc_k^N v^N + g_k)\d W^k + (S^N(v^N)+f)\,\d t,\\
			& v^N(s)=v_s^N\in L^p((\Omega,\Fc_s,\Pb);X_{tr}),
		\end{aligned}
		\right.
	\end{equation}
	where $s\geq 0$. Fix $T\in (s,\infty]$ and a stopping time $\tau:\Omega\rightarrow [s,T]$. A progressively measurable process $v^N: \Omega\times[s,\tau]\rightarrow X_1$ is said to be a solution to \eqref{Stokes-v} on $(s,\tau)$ provided $v^N\in L^p((s,\tau); X_1),\Pb\text{-a.s.}$ and, $\Pb\text{-a.s.}$ for every $t\in (s,\tau)$,
	\begin{equation*}
		v^N(t)=v_s^N+\int_{s}^{t} \big( \Delta v^N + S^N(v^N) + f\big)\,\d r + \sum_k\int_{s}^{t} \mathbb{I}_{[0,\tau]}(\Lc_k^N v^N + g_k)\d W_r^k.
	\end{equation*} 
	
	By adapting the arguments from \cite[Theorem 3.1]{Agr26}, we establish uniform maximal $L^p$-regularity estimates for the stochastic Stokes equations \eqref{Stokes-v}.
	\begin{proposition}\label{USMR}
		Fix $a\in [-1,0)$ and $p\in [2,\infty)$. Let $s,T,\tau$ be defined as above. Then there exists a unique solution to \eqref{Stokes-v} on $(s,\tau)$, which satisfies the following estimate: 
		\begin{itemize}
			\item[(1).] If $T<\infty$, then there exists $C_1=C_1(\nu,p,T-s)$ such that
			\begin{align*}
				&\E\bigg[ \sup_{t\in [s,\tau]}\|v^N(t)\|_{X_{tr}}^p\bigg] + \E\bigg[ \int_{s}^{\tau}\|v^N(t)\|_{X_{1}}^p\,\d t\bigg]\\
				&\leq C_1 \E\Big[\|v_s^N\|_{X_{tr}}^p\Big] + C_1\E\bigg[\int_{s}^{\tau}\Big(\|f(t)\|_{X_0}^p + \|g(t)\|_{L^p(\gamma(l^2;X_{1/2})}^p\Big)\,\d t\bigg].
			\end{align*}
			\item[(2).] If $v^N\in L^p(\Omega\times (s,\tau);X_{0})$, then there exists $C_2=C_2(\nu,p)$ such that
			\begin{align*}
				&\E\bigg[ \sup_{t\in [s,\tau]}\|v^N(t)\|_{X_{tr}}^p\bigg] + \E\bigg[ \int_{s}^{\tau}\|v^N(t)\|_{X_{1}}^p\,\d t\bigg]\\
				&\leq  C_1\E\bigg[\int_{s}^{\tau}\Big(\|f(t)\|_{X_0}^p + \|g(t)\|_{L^p(\gamma(l^2;X_{1/2})}^p\Big)\,\d t\bigg]\\ 
				&\quad+C_2 \E\Big[\|v_s^N\|_{X_{tr}}^p\Big] + C_2\E\bigg[ \int_{s}^{\tau}\|v^N(t)\|_{X_{0}}^p\,\d t\bigg].
			\end{align*}
		\end{itemize}
	\end{proposition}
	Note that for the case $p=2$, the anti-symmetry of the noises on $X_{1/2}$ allows for $N$-uniform estimates by the energy method in the previous section. It is precisely the viscosity term that makes it possible to get uniform estimates in $L^p(\mathbb{R}_+; X_1)$ for $p > 2$. We also mention that the above also holds with additional optimal space-time regularity estimates in the spaces $L^p(\Omega;H^{\theta,p}((s,\tau);X_{1-\theta}))$ for every $\theta\in [0,\frac{1}{2})$.
	\begin{proof}
		It suffices to prove (2), as the proof of (1) is similar and simpler.
		
		\noindent\textbf{Special case: $\sigma_k^N\equiv 0$.}  Note that the operator $-\Delta: X_1\rightarrow X_0$ is invertible and has a bounded $H^\infty-$calculus with angle less than $\frac{\pi}{2}$ (c. f. \cite[Example 3.2]{VNVW12}). Hence by stochastic maximal $L^p-$regularity \cite[Theorem 3.5]{VNVW12}, there exists a unique solution $v^N$ to \eqref{Stokes-v}, which satisfies,
		\begin{align*}
			& \|v^N\|_{L^p(\Omega;C((s,\tau);X_{tr}))} +  \|v^N\|_{L^p(\Omega\times(s,\tau);X_{1})}\\
			&\lesssim_{p} \|v_s\|_{L^p(\Omega;X_{tr})} +\|f\|_{L^p(\Omega\times(s,\tau);X_0)} + \|g\|_{L^p(\Omega\times(s,\tau);\gamma(l^2;X_{1/2}))}.
		\end{align*}
		\textbf{General case.} Let $\Lc^N=(\Lc_k^N)_k$. We claim that, for every $\epsilon\in (0,1)$, there exists $C_\epsilon>0$ independent of $N$, such that , for every $v\in X_1$,
		\begin{equation}\label{claim}
			\|S^N(v)\|_{X_0} + \|\Lc^N v\|_{\gamma(l^2,X_{1/2})}\leq \epsilon \|v\|_{X_1} + C_\epsilon \|v\|_{X_0}.
		\end{equation}
		Accepting this claim for now, the existence and the uniqueness of the solution to \eqref{Stokes-v} follows from \cite[Theorem 3.2]{AV24}. Then regarding $S^N$ and $\Lc^N$ as a perturbation and substituting \eqref{claim} into the estimates in \textbf{Special case}, we obtain
		\begin{align*}
			&\|v^N\|_{L^p(\Omega;C((s,\tau);X_{tr}))} +  \|v^N\|_{L^p(\Omega\times(s,\tau);X_{1})}\\
			&\lesssim_{p} \|v_s\|_{L^p(\Omega;X_{tr})} +\|f\|_{L^p(\Omega\times(s,\tau);X_0)} + \|g\|_{L^p(\Omega\times(s,\tau);\gamma(l^2;X_{1/2}))}  \\
			&\quad + \|S^N(v^N)\|_{L^p(\Omega\times(s,\tau);X_0)} + \|\Lc^N v^N\|_{L^p(\Omega\times(s,\tau);\gamma(l^2;X_{1/2}))} \\
			&\lesssim_{p} \|v_s\|_{L^p(\Omega;X_{tr})} +\|f\|_{L^p(\Omega\times(s,\tau);X_0)} + \|g\|_{L^p(\Omega\times(s,\tau);\gamma(l^2;X_{1/2}))}  \\
			&\quad + C_\epsilon  \|v^N\|_{L^p(\Omega\times (s,\tau);X_{0})} + \epsilon \|v^N\|_{L^p(\Omega\times (s,\tau);X_{1})}.
		\end{align*}
		The last term on the RHS can be absorbed into the LHS, provided that $\epsilon$ is chosen small enough.
		
		It remains to verify the claim \eqref{claim}. By Lemma \ref{order-S^N}, for every $\epsilon\in (0,1)$, we have
		\begin{align*}
			\|S^N(v)\|_{X_0}^2 &=\sum_{\xi} 	\big|\widehat{S^{N}(v)}(\xi)\big|^2\,|\xi|^{2(2b+a-1)} \\
			&\leq \epsilon \sum_{\xi} |\hat{v}(\xi)|^2\,|\xi|^{2(2b+a+1)} + C_\epsilon \sum_{\xi} |\hat{v}(\xi)|^2\,|\xi|^{2(2b+3a+1)} \\
			&= \epsilon \|v\|_{X_1}^2 + C_\epsilon \|v\|_{X_{1+a}}^2.
		\end{align*}
		Since $-1\leq a<0$ by assumption, by the interpolation $X_0\cap X_1\hookrightarrow X_{1+a}$, the claim holds for $S^N$. On the other hand, by the isomorphism \eqref{isomorphism}, we have
		\begin{align*}
			\|\Lc^N v\|_{\gamma(l^2,X_{1/2})}^2 &\simeq  \sum_k \big\<\Lc_k^N v, (-\Delta)^{2b+a}\Lc_k^N v\big\> \\
			& =\sum_k \big\<\sigma_k^N\cdot\nabla(-\Delta)^{a+b}v ,(-\Delta^a)\Pi\big(\sigma_k^N\cdot\nabla(-\Delta)^{a+b}v\big)\big\>\\
			& =-\big\<(-\Delta)^{2b+a}v ,S^N(v)\big\>.
		\end{align*}
		Taking Fourier transform and using Lemma \ref{order-S^N}, we get, for every $\epsilon\in (0,1)$,
		\begin{align*}
			\|\Lc^N v\|_{\gamma(l^2,X_{1/2})}^2 &\lesssim \sum_\xi \big|\widehat{S^{N}(v)}(\xi)\big|^2\,|\xi|^{2(2b+a)}  \\
			&\leq \epsilon \sum_\xi |\hat{v}(\xi)|^2|\xi|^{2(2b+a+1)} + C_\epsilon  \sum_\xi |\hat{v}(\xi)|^2|\xi|^{2(2b+2a+1)} \\
			&= \epsilon \|v\|_{X_1}^2 + C_\epsilon \|v\|_{X_{1+a/2}}^2.
		\end{align*}
		The desired claim follows from the interpolation $X_0\cap X_1\hookrightarrow X_{1+a/2}$.
	\end{proof}
	
	\subsection{Application to 3D NS equations}
	Let $a\in [-1,0)$. Consider the stochastic NS equations:
	\begin{equation}\label{NS-SMR section}
		\left\{
		\begin{aligned}
			&\d u^N +  
			\Pi(u^N\cdot\nabla u^N)\,\d t
			=\sum_k  \Lc_k^N u^N \circ \d W^k + \Delta u^N\,\d t,\\
			&u^N(0)=u_0\in X_{tr}=B_{2,p}^{2b+a+1-\frac{2}{p}}.
		\end{aligned}
		\right.
	\end{equation}
	
	The definition of strong solutions on time interval $[0,T]$ to \eqref{NS-SMR section} follows Definition \ref{NS-def sol}, except that the regularity space in (1) is replaced by $C([0,T];X_{tr}) \cap L^p(0,T;X_1)$. 
	
	\begin{theorem}\label{SMR-finite time-delayed blow-up}
		Let $a\in [-1,0)\text{ and }2b+a>-\frac{1}{2}$. Fix $T>0, \frac{4}{2(2b+a)+1}<p<\infty$,
		\begin{equation*}
			0<\epsilon<1 \quad \text{ and }\quad R_0>0.
		\end{equation*} 
		Then there exists $\nu_0=\nu_0(R_0)\text{ and }N_0=N_0(\nu,R_0,\epsilon,T)$, such that for all $\nu>\nu_0,\,N>N_0$ and initial data $u_0$ satisfying
		\begin{equation*}
			u^N(0)=u_0\in X_{tr}\quad \text{ and } \quad\|u_0\|_{X_{tr}}\leq R_0,
		\end{equation*}
		the equation \eqref{NS-SMR section} admits a unique strong solution (which is spatially smooth) that exists on the interval $[0, T]$ with probability at least $1-\epsilon$.
	\end{theorem}
	\begin{proof}
		Clearly, both the spaces $C([0,T];X_{tr})$ and $L^p(0,T;X_1)$ are subcritical under the assumptions of the theorem.  Let $R>0$, $r\in (\frac{1}{2},2b+a+1-\frac{2}{p})$ and consider the stochastic NS equations with cut-off\footnote{Similarly to Section 3.2, we slightly abuse notation by letting $u^N$ represent the solutions of both the original system and its truncation.}:
		\begin{equation*}
			\left\{
			\begin{aligned}
				&\d u^N +  
				f_R(\|u^N\|_{H^r})\Pi(u^N\cdot\nabla u^N)\,\d t
				=\sum_k  \Lc_k^N u^N \circ \d W^k + \Delta u^N\,\d t,\\
				&u^N(0)=u_0\in X_{tr}=B_{2,p}^{2b+a+1-\frac{2}{p}}.
			\end{aligned}
			\right.
		\end{equation*}
		As the norm $H^r$ used for truncation is still subcritical, the maximal regularity estimates from Proposition \ref{USMR}(1) yield $N$-uniform estimates, following a derivation analogous to that of \cite[Theorem 6.2]{Agr26}: for every $T\in (0,\infty)$,
		\begin{equation}\label{bound1-SMR}
			\|u^N\|_{L^p(\Omega;C([0,T];X_{tr}))} +  \|u^N\|_{L^p(\Omega\times[0,T];X_{1})} \lesssim_{R,T} 1 + \|u_0\|_{X_{tr}}.
		\end{equation}
		The bound \eqref{bound1-SMR} plays the same role of \eqref{bound1} in this context. On the other hand, since $X_{tr} \hookrightarrow X_{1/2} = H^{2b+a}$ for $p > 2$, one can readily check from Lemma \ref{bound2} that the uniform bound \eqref{time-C^0.5} also holds for $\{u^N\}$. By virtue of the uniform bounds \eqref{bound1-SMR},\,\eqref{bound2} and the scaling limit in Lemma \ref{S^N limit}, one can show that the limit equation is 
		\begin{equation*}
			\partial_t u + f_R(\|u\|_{H^r})\Pi(u\cdot\nabla u) = \Delta u - \frac{3\nu}{5}(-\Delta)^{1+a}.
		\end{equation*}
		The solution properties are listed in Proposition \ref{limit-NS}.
		Then the delayed blow-up on finite time intervals can be established similarly to \textbf{Part 1} of the proof of Theorem ~\ref{Delay blow up-NS}.
	\end{proof}
	
	As shown in \textbf{Part 2} of Theorem \ref{Delay blow up-NS}, the key to extending the solution's lifetime to infinity lies in establishing ($N$-uniform) small-data global well-posedness for \eqref{NS-SMR section}. In the special case $2b+a=0$, where both the noise and the nonlinear terms preserve the $L^2$-norm, we obtain the following global control: $\Pb\text{-a.s.}$, for every $t\geq 0$,
	\begin{equation*}
		\|u^N(t)\|_{L^2}^2 + 2\int_{0}^{t}\|\nabla u^N(s)\|_{L^2}^2\,\d s = \|u_0\|_{L^2}^2,
	\end{equation*}
	which, combined with the Poincar\'e inequality and the Gr\"onwall inequality, yields the exponential decay 
	\begin{equation}\label{exp-decay}
		\|u^N(t)\|_{L^2}\leq e^{-t}\|u_0\|_{L^2}\lesssim_p e^{-t}\|u_0\|_{X_{tr}},
	\end{equation}
	where we have used $X_{tr}=B_{2,p}^{1-\frac{2}{p}}\hookrightarrow L^2$ when $p\geq2$. We note that, in view of the exponential decay \eqref{exp-decay}, Proposition \ref{USMR} (2) can be applied to $u^N$. In this context, the lower-order term on the RHS can be absorbed into the initial value $\|u_0\|_{X_{tr}}$. Proceeding exactly as in the proof of \cite[Theorem 4.3]{Agr26}, we obtain the following proposition (which corresponds to the case $\gamma=1$ in \cite[Theorem 4.3]{Agr26}): 
	\begin{proposition}[Global well-posedness--small data in critical space]
		Let $a\in [-1,0),\,2b+a =0\text{ and } p=4$. For every $\epsilon\in (0,1)$, there exists $\delta=\delta(\nu,\epsilon)>0$ such that for all $s\geq0$ and \begin{equation*}
			u^N(s)\in L^p((\Omega,\Fc_s,\Pb);X_{tr}) \quad \text{ and }\quad \|u^N(s)\|_{L^p(\Omega,X_{tr})}\leq \delta,
		\end{equation*}
		the local strong solution $u^N$ to \eqref{NS-SMR section} is global on $[s,\infty)$ with probability no less than $1-\epsilon$. 
	\end{proposition}
	The key point of the above proposition is that the choice of $\delta$ is independent of $N$. The proof of the following theorem is identical to that of \cite[Theorem 2.2]{Agr26}; we thus refer the reader to \cite[Section 6.1]{Agr26} for further details.
	\begin{theorem}\label{SMR-delayed blow-up}
		Let $a\in [-1,0)\text{ and }2b+a =0$. Fix $4<p<\infty$,
		\begin{equation*}
			0<\epsilon<1 \quad \text{ and }\quad R_0>0.
		\end{equation*} 
		Then there exists $\nu_0=\nu_0(R_0)\text{ and }N_0=N_0(\nu,R_0,\epsilon)$, such that for all $\nu>\nu_0,\,N>N_0$ and initial data $u_0$ satisfying
		\begin{equation*}
			u^N(0)=u_0\in X_{tr}\quad \text{ and }\quad \|u_0\|_{X_{tr}}\leq R_0,
		\end{equation*}
		the equation \eqref{NS-SMR section} admits a unique strong solution which is spatially smooth and, with a probability no less than $1-\epsilon$, exists globally on $[0,\infty)$.
	\end{theorem}
	\section*{Acknowledgements}
	The authors would like to thank Dejun Luo for helpful discussions during the early stages of this project. The work of the first author is supported by the China Scholarship Council (Grant No. 202504910286). The second author acknowledges the partial support of the project PNRR - M4C2 - Investimento 1.3, Partenariato Esteso PE00000013 - \emph{FAIR - Future Artificial Intelligence Research} - Spoke 1 \emph{Human-centered AI}, funded by the European Commission under the NextGeneration EU programme, of the project \emph{Noise in fluid dynamics and related models} funded by the MUR Progetti di Ricerca di Rilevante Interesse Nazionale (PRIN) Bando 2022 - grant 20222YRYSP, and of the MUR Excellence Department Project awarded to the Department of Mathematics, University of Pisa, CUP I57G22000700001.

	\appendix
	
	\section{Proof of the commutator estimates (\ref{commutator1})-(\ref{commutator2})}\label{proof of commutator}
	First we recall the definition of the standard mollification operator $J_\epsilon$ on $\Tb^3$. Let $B(x,r)$  denote the open ball centered at $x$ and radius $r$. Fix a standard mollifying kernel $\rho\in C_c^\infty(B(0,1))$, which is a non-negative and radial function satisfying $\int_{\R^3}\rho(x)\,\d x=1$. For each $\epsilon\in (0,1)$, define $\rho_\epsilon(\cdot)=\epsilon^{-3}\rho(\epsilon^{-1}\cdot)$. The function $\rho_\epsilon$ induces a periodic function $\rho_\epsilon^{\text{per}}$, defined as
	\begin{equation*}
		\rho_\epsilon^{\text{per}}(\cdot)=\sum_{k\in \Z^3}\rho_\epsilon(\cdot-2\pi k).
	\end{equation*} 
	By Poisson summation formula (c.f. \cite[p.~254, 8.32]{Fol99}), the Fourier series of $\rho_\epsilon^{\text{per}}$ coincides with the Fourier transform of $\rho_\epsilon$, that is,
	\begin{equation*}
		\Fc\big[\rho_\epsilon^{\text{per}}\big](k)=\widehat{\rho_\epsilon}(k)=\hat{\rho}(\epsilon k),\quad \text{ for every } k\in\Z^3.
	\end{equation*}
	For every $f\in \mathcal{S}'(\Tb^3)$, define $J_\epsilon f=\rho_\epsilon^{\text{per}}\ast f$. We note that $\rho_\epsilon$ is supported in $B(0,\epsilon)\subset B(0,1) $, implying that for every $x\in \Tb^3$,
	\begin{align*}
		J_\epsilon f(x)&=\sum_{k\in \Z^3}\int_{[-\pi,\pi)^3}\rho_\epsilon(y-2\pi k)f(x-y)\,\d y\\ &=\int_{[-\pi,\pi)^3}\rho_\epsilon(y)f(x-y)\,\d y= \int_{\R^3}\rho_\epsilon(y)f(x-y)\,\d y=\rho_\epsilon\ast f(x).
	\end{align*}
	Henceforth, we identify $\rho_\epsilon^{\text{per}}$ with $\rho_\epsilon$, given that the convolution operators induced by them are equivalent.
	\begin{proof}[Proof of \eqref{commutator1}]
		First we note that
		\begin{equation*}
			[J_\epsilon\Lambda^n,u\cdot\nabla]  = J_\epsilon[\Lambda^n,u\cdot\nabla] +
			[J_\epsilon,u\cdot\nabla]\Lambda^n,
		\end{equation*}
		which implies
		\begin{align*}
			\big\<[J_\epsilon\Lambda^n,u\cdot\nabla] u, J_\epsilon\Lambda^n u\big\>
			&= \big\<J_\epsilon[\Lambda^n,u\cdot\nabla]u, J_\epsilon\Lambda^n u\big\> + \big\<[J_\epsilon,u\cdot\nabla]\Lambda^nu, J_\epsilon\Lambda^n u\big\>\\
			&=:I_1^\epsilon+I_2^\epsilon.
		\end{align*}
		It is obvious that as $\epsilon\rightarrow 0$, $I_1^\epsilon\rightarrow \big\<[\Lambda^n,u\cdot\nabla]u, \Lambda^n u\big\>$, hence it remains to prove $\lim\limits_{\epsilon\rightarrow 0}I_2^\epsilon=0$. By the Minkowski inequality  and  Sobolev embedding $H^{n+1}\hookrightarrow C^{\alpha}$, where $\alpha=\min\{n-1/2,1\}$, we have  
		\begin{equation*}
			\begin{split}
				\|[J_\epsilon,u\cdot\nabla]\Lambda^nu \|_{L^2(\Tb^3)} &= \bigg\|\int_{\Tb^3} \rho_\epsilon (y)\big(u(x-y)-u(x)\big)\cdot \nabla \Lambda^nu (x-y)\,\d y\bigg\|_{L_x^2}\\
				&\leq \int_{\Tb^3} |\rho_\epsilon(y)|\, \big\|\big(u(x-y)-u(x)\big) \cdot \nabla \Lambda^nu (x-y)\big\|_{L_x^2}\,\d y\\
				&\lesssim_{n} \|u\|_{H^{n+1}} \int_{\Tb^3} \rho_\epsilon (y) |y|^{\alpha}\|\nabla\Lambda^n u(x-y)\|_{L_x^2}\,\d y.
			\end{split}
		\end{equation*}
		Since the integral is invariant under translations, we get 
		\begin{equation*}
			\|[J_\epsilon,u\cdot\nabla]\Lambda^nu \|_{L^2(\Tb^3)}\lesssim_n \|u\|_{H^{n+1}}^2 \int_{\R^3} \rho_\epsilon (y) |y|^{\alpha}\,\d y\lesssim_n \epsilon^{\alpha}\|u\|_{H^{n+1}}^2.
		\end{equation*}
		Thus we obtain
		\begin{equation*}
			|I_2^\epsilon|\leq 	\|[J_\epsilon,u\cdot\nabla]\Lambda^nu \|_{L^2} \|J_\epsilon\Lambda^n u\|_{L^2}\lesssim_n \epsilon^{\alpha} \|u\|_{H^{n+1}}^2\|u\|_{H^{n}},
		\end{equation*}
		which proves the claim as $u\in C([0,T];H^{n})\cap L^2(0,T;H^{n+1}),\,\Pb\text{-a.s.}$
	\end{proof}
	
	Next we turn to the proof of \eqref{commutator2}. When estimating noise, general cases reduce to the special case $b=0$.
	\begin{proof}[Proof of \eqref{commutator2}]
		With loss of generality, assume $b=0$, then $n=2b+a=a$ and $\Lc_ku=\Pi\big(\sigma_k\cdot\nabla(-\Delta)^a u\big)$. As $\Pi$ commutes with $J_\epsilon$, we can write
		\begin{equation*}
			S_k^\epsilon u= \big[J_\epsilon,\Lc_k\big]u=\Pi\big[J_\epsilon,\sigma_k\cdot\nabla\big](-\Delta)^a u.
		\end{equation*}
		Suppose $P$ is a pseudo-differential operator of order $m$, whose symbol belongs to the standard H\"ormander class $S_{1,0}^m$ (c.f. \cite{Hor07}). By the classical theory, for every $s\in \R$ and $\theta\in [0,1]$,
		\begin{equation}\label{mollify-commutator}
			\|\big[J_\epsilon,P\big]f\|_{H^s}\lesssim_{k,s}\epsilon^\theta\|f\|_{H^{s+m-1+\theta}}.
		\end{equation} 
		The estimate $\|S_k^\epsilon\|_{H^a}\lesssim_{k} \epsilon^{1-2a} \|u\|_{H^{a+1}}$ follows immediately by \eqref{mollify-commutator}. 
		
		The commutator $T_k^\epsilon=\big[S_k^\epsilon,\Lc_k\big]$ can be split as follows:
		\begin{equation*}
			\begin{split}
				T_k^\epsilon  &=\big[\Pi\big[J_\epsilon,\sigma_k\cdot\nabla\big](-\Delta)^{a},\Pi(\sigma_k\cdot\nabla)(-\Delta)^{a}\big]\\
				&=\Pi \big[J_\epsilon,\sigma_k\cdot\nabla\big]\:\big[(-\Delta)^{a},\Pi(\sigma_k\cdot\nabla)\big]\:(-\Delta)^a \\
				&\quad+ \big[\Pi\big[J_\epsilon,\sigma_k\cdot\nabla\big],\Pi(\sigma_k\cdot\nabla)\big]\: (-\Delta)^{2a} \\
				&\quad+ \Pi(\sigma_k\cdot\nabla)\:\big[\Pi\big[J_\epsilon,\sigma_k\cdot\nabla\big],(-\Delta)^{a}\big]\:(-\Delta)^{a}\\
				&=:R_1^\epsilon + R_2^\epsilon + R_3^\epsilon.
			\end{split}
		\end{equation*} 
		Recall that our goal is to prove $\|T_k^\epsilon\|_{H^{-a}}\lesssim_k\epsilon^{1-2a}\|u\|_{H^{a+1}}$.
		\begin{itemize}
			\item [ (1).] Noting that the commutator $\big[(-\Delta)^{a},\Pi(\sigma_k\cdot\nabla)\big]$ is of order $2a$, it follows directly from \eqref{mollify-commutator} that $\|R_1^\epsilon\|_{H^{-a}}\lesssim_k\epsilon^{1-2a}\|u\|_{H^{a+1}}$.
			\item [(2).] We have
			\begin{equation*}
				\Pi\big[J_\epsilon,\sigma_k\cdot\nabla\big]= \big[J_\epsilon,\Pi(\sigma_k\cdot\nabla)\big],
			\end{equation*}
			yielding 
			\begin{equation*}
				R_2^\epsilon= \big[\big[J_\epsilon,\Pi(\sigma_k\cdot\nabla)\big],\Pi(\sigma_k\cdot\nabla)\big]\:(-\Delta)^{2a}.
			\end{equation*}
			Regarding the double commutator,  for every $s\in \R$ and $\theta\in [0,2]$, the following holds:
			\begin{equation}\label{double-commutator}
				\big\|\big[\big[J_\epsilon,\Pi(\sigma_k\cdot\nabla)\big],\Pi(\sigma_k\cdot\nabla)\big]f\big\|_{H^s}\lesssim_{k,s} \epsilon^\theta\|f\|_{H^{s+\theta}}.
			\end{equation}
			The estimate $\|R_2^\epsilon\|_{H^{-a}}\lesssim_k\epsilon^{1-2a}\|u\|_{H^{a+1}}$ is a direct consequence of \eqref{double-commutator}.
			\item[(3).] As $(-\Delta)^a$ commutes with $\Pi$, we get
			\begin{equation*}
				\big[\Pi\big[J_\epsilon,\sigma_k\cdot\nabla\big],(-\Delta)^{a}\big] = \Pi\big[\big[J_\epsilon,\sigma_k\cdot\nabla\big],(-\Delta)^{a}\big].
			\end{equation*}
			By Jacobi identity, we have 
			\begin{equation*}
				\big[\big[J_\epsilon,\sigma_k\cdot\nabla\big],(-\Delta)^{a}\big]=-\big[J_\epsilon,\big[(-\Delta)^{a},\sigma_k\cdot\nabla\big]\big],
			\end{equation*}
			hence we obtain
			\begin{equation*}
				I_3=-\Pi(\sigma_k\cdot\nabla)\:\big[J_\epsilon,\big[(-\Delta)^{a},\sigma_k\cdot\nabla\big]\big]\:(-\Delta)^{a}.
			\end{equation*}
			As $\big[(-\Delta)^{a},\sigma_k\cdot\nabla\big]$ is of order $2a$, it is easy to check that $\|R_3^\epsilon\|_{H^{-a}}\lesssim_k\epsilon^{1-2a}\|u\|_{H^{a+1}}$ follows from \eqref{mollify-commutator}.
		\end{itemize}
		To conclude the proof, we only need to show the validity of the double commutator estimate \eqref{double-commutator}. By Riesz-Thorin interpolation \cite[p.~200, 6.27]{Fol99}, it suffices to prove the endpoint cases $\theta=0$ and $\theta=1$. Since the expression involves the Leray projection, we shall perform estimates in frequency space. For simplicity of notation, let $M=\sigma_k\cdot\nabla$ and $U_\epsilon =\big[\big[J_\epsilon,M\big],M\big]$, then we have for every $\xi\in\Z_0^3$,
		\begin{align*}
			\widehat{Mf}(\xi)&= (2\pi)^{-\frac{3}{2}}P_\xi^\perp\sum_{\eta}\widehat{\sigma_k}(\xi-\eta)\cdot i\eta \hat{f}(\eta), \\
			\widehat{M^2f}(\xi)&=-(2\pi)^{-3}P_\xi^\perp\sum_{\eta,\zeta}\widehat{\sigma_k}(\xi-\zeta)\cdot\zeta\; \widehat{\sigma_k}(\zeta-\eta)\cdot\eta\; P_\zeta^\perp\hat{f}(\eta).
		\end{align*}
		For the double commutator $U_\epsilon$, we have
		\begin{align*}
			U_\epsilon=\big[\big[J_\epsilon,M\big],M\big] &=\big(J_\epsilon M-MJ_\epsilon\big)M-M\big(J_\epsilon M-MJ_\epsilon\big) \\
			&= J_\epsilon M^2-2MJ_\epsilon M+ M^2J_\epsilon.
		\end{align*}
		Taking Fourier transform and neglecting $-(2\pi)^{-3}P_\xi^\perp$ , we get
		\begin{align*}
			\widehat{U_\epsilon f}(\xi)&\approx \widehat{\rho_\epsilon}(\xi)\sum_{\eta,\zeta}\widehat{\sigma_k}(\xi-\zeta)\cdot\zeta\; \widehat{\sigma_k}(\zeta-\eta)\cdot\eta\; P_\zeta^\perp\hat{f}(\eta) \\
			&\quad-2 \sum_{\eta,\zeta}\widehat{\sigma_k}(\xi-\zeta)\cdot\zeta\; \widehat{\rho_\epsilon}(\zeta) \widehat{\sigma_k}(\zeta-\eta)\cdot\eta\; P_\zeta^\perp\hat{f}(\eta) \\
			&\quad+ \sum_{\eta,\zeta}\widehat{\sigma_k}(\xi-\zeta)\cdot\zeta\;  \widehat{\sigma_k}(\zeta-\eta)\cdot\eta\; \widehat{\rho_\epsilon}(\eta)P_\zeta^\perp\hat{f}(\eta) \\
			&=\sum_{\eta,\zeta} \widehat{\sigma_k}(\xi-\zeta)\cdot\zeta\; \widehat{\sigma_k}(\zeta-\eta)\cdot\eta \big(\widehat{\rho_\epsilon}(\xi)-2\widehat{\rho_\epsilon}(\zeta)+\widehat{\rho_\epsilon}(\eta)\big)P_\zeta^\perp\hat{f}(\eta).
		\end{align*}
		Now let $m=\xi-\zeta,\,l=\zeta-\eta$, then $\eta=\xi-(m+l)$. The above equation can be written as
		\begin{equation}\label{U-epsilon}
			\begin{split}
				\widehat{U_\epsilon f}(\xi)&=-(2\pi)^{-3}P_\xi^\perp\sum_{m,l}\widehat{\sigma_k}(m)\cdot(\eta+l)\; \widehat{\sigma_k}(l)\cdot\eta\; P_\zeta^\perp\hat{f}(\eta)\\
				&\quad \times\big(\widehat{\rho_\epsilon}(\eta+m+l)-2\widehat{\rho_\epsilon}(\eta+l)+\widehat{\rho_\epsilon}(\eta)\big).
			\end{split}
		\end{equation}
		Note that the kernel $\Phi_\epsilon=\big(\widehat{\rho_\epsilon}(\eta+m+l)-2\widehat{\rho_\epsilon}(\eta+l)+\widehat{\rho_\epsilon}(\eta)\big)$ is not the standard second-order difference $\Psi_\epsilon=\big(\widehat{\rho_\epsilon}(\eta+m+l)-\widehat{\rho_\epsilon}(\eta+l)-\widehat{\rho_\epsilon}(\eta+m)+\widehat{\rho_\epsilon}(\eta)\big)$. By replacing the former with the latter, we arrive at a new operator $V_\epsilon$ defined by
		\begin{equation}\label{V-epsilon}
			\begin{split}
				\widehat{V_\epsilon f}(\xi)&=-(2\pi)^{-3}P_\xi^\perp\sum_{m,l}\widehat{\sigma_k}(m)\cdot(\eta+l)\; \widehat{\sigma_k}(l)\cdot\eta\; P_\zeta^\perp\hat{f}(\eta)\\
				&\quad \times \big(\widehat{\rho_\epsilon}(\eta+m+l)-\widehat{\rho_\epsilon}(\eta+l)-\widehat{\rho_\epsilon}(\eta+m)+\widehat{\rho_\epsilon}(\eta)\big).
			\end{split}
		\end{equation}
		In what follows, we establish the estimate \eqref{double-commutator} for $V_\epsilon$ and $U_\epsilon-V_\epsilon$ in turn, which yields the desired assertion. The standard second-order difference $\Psi$ satisfies
		\begin{equation*}
			\begin{split}
				\Psi_\epsilon&=
				\int_{0}^{1}\int_{0}^{1} \frac{\d^2}{\d t_1\d t_2} \widehat{\rho}(\epsilon(\eta+t_1m+t_2l))\,\d t_1\d t_2\\
				&=\epsilon^2 \int_{0}^{1}\int_{0}^{1}  \nabla^2\widehat{\rho}(\epsilon(\eta+t_1m+t_2l)):m\otimes l\,\d t_1\d t_2.
			\end{split}
		\end{equation*}
		By assumption, $\hat{\rho}$ is a Schwartz function, yielding
		\begin{align}
			|\Psi_\epsilon|&\lesssim\epsilon^2 |m|\,|l|,\label{epsilon^2}\\
			|\eta+l|\,|\eta|\, |\Psi_\epsilon|&\lesssim |m|\,|l|\big(1 + \epsilon^2(|l|^2 +|m|^2)\big) \lesssim |m|^3|l|^3,\label{epsilon^0}
		\end{align}
		where the underlying constant is independent of $\epsilon\in (0,1)$ and $\eta,m,l\in \Z_0^3$. 
		Substituting \eqref{epsilon^2} into \eqref{V-epsilon}, we get
		\begin{align*}
			\big|\widehat{V_\epsilon f}(\xi)\big|&\lesssim \epsilon^2\sum_{m,l}|m|\,|\widehat{\sigma_k}(m)|\,|l|\,|\widehat{\sigma_k}(l)|\;|\eta+l|\; |\eta|\; |\hat{f}(\eta)| \\
			&\lesssim \epsilon^2\sum_{m,l}|m|\,|\widehat{\sigma_k}(m)|\,|l|\,|\widehat{\sigma_k}(l)|\; |\eta|^2\; |\hat{f}(\eta)| + \epsilon^2\sum_{m,l}|m|\,|\widehat{\sigma_k}(m)|\,|l|^3\,|\widehat{\sigma_k}(l)|\;  |\hat{f}(\eta)|.
		\end{align*}
		Recalling $m=\xi-\zeta$ and $l=\zeta-\eta$, both terms above are of convolution type. Hence by Young's inequality on $L^p$ with weight $|\cdot|^{2s}$, we obtain 
		\begin{equation}\label{V-epsilon2}
			\|V_\epsilon f\|_{H^s}\lesssim_{k,s} \epsilon^2 \|f\|_{H^{s+2}}.
		\end{equation}
		By estimate \eqref{epsilon^0}, analogous to the derivation above, we obtain 
		\begin{equation}\label{V-epsilon0}
			\|V_\epsilon f\|_{H^s}\lesssim_{k,s} \|f\|_{H^{s}}.
		\end{equation}
		Combining Riesz-Thorin interpolation and \eqref{V-epsilon2}-\eqref{V-epsilon0}, we have proved the estimate \eqref{double-commutator} for $V_\epsilon$. It remains to estimate the difference $U_\epsilon-V_\epsilon$ to complete the proof. By definition, still neglecting $-(2\pi)^{-3}P_\xi^\perp$, we have
		\begin{equation*}
			\begin{split}
				\widehat{(U_\epsilon-V_\epsilon) f}(\xi)&\approx\sum_{m,l}\widehat{\sigma_k}(m)\cdot(\eta+l)\; \widehat{\sigma_k}(l)\cdot\eta\;(\Phi_\epsilon-\Psi_\epsilon) P_\zeta^\perp\hat{f}(\eta)\\
				&=\sum_{m,l}\widehat{\sigma_k}(m)\cdot l\; \widehat{\sigma_k}(l)\cdot\eta\;(\Phi_\epsilon-\Psi_\epsilon) P_\zeta^\perp\hat{f}(\eta)  \\
				&\quad+ \sum_{m,l}\widehat{\sigma_k}(m)\cdot \eta\; \widehat{\sigma_k}(l)\cdot\eta\;(\Phi_\epsilon-\Psi_\epsilon) P_\zeta^\perp\hat{f}(\eta)\\
				&=: R_4^\epsilon+R_5^\epsilon.
			\end{split}
		\end{equation*}
		As $\widehat{\rho_\epsilon}$ is a even function, we have $\big|\nabla\widehat{\rho_\epsilon}(\eta)\big|=\epsilon|\nabla\widehat{\rho}(\epsilon\eta)|\lesssim \epsilon^2 |\eta|$ for every $\eta$. Hence we have
		\begin{equation*}
			\big|\Phi_\epsilon-\Psi_\epsilon\big|=\big|\widehat{\rho_\epsilon}(\eta+m)-\widehat{\rho_\epsilon}(\eta+l)\big|\lesssim \epsilon^2\big(|m|^2+|l|^2+|\eta|\,|m-l|\big)
		\end{equation*}
		Then arguing as \eqref{V-epsilon2}, we obtain $\|R_4^\epsilon\|_{H^s}\lesssim_k\epsilon^2\|f\|_{H^{s+2}}$. Besides, for every $\epsilon\in (0,1)$ and $\eta,m,l\in \Z_0^3$, we have
		\begin{align*}
			|\eta| \,	\big|\Phi_\epsilon-\Psi_\epsilon\big|
			&=|\eta|\, \bigg|\int_{0}^{1} \frac{\d}{\d t}\widehat{\rho_\epsilon}(\eta+tm+(1-t)l)\,\d t\bigg|  \\
			&\leq \epsilon|\eta|\,|m-l|\, \bigg|\int_{0}^{1} \nabla\widehat{\rho}(\epsilon(\eta+tm+(1-t)l))\,\d t\bigg| \\
			&\lesssim |m|^2+|l|^2.
		\end{align*}
		Then arguing as \eqref{V-epsilon0}, we get $\|R_4^\epsilon\|_{H^s}\lesssim_k\|f\|_{H^{s}}$. 
		
		It remains to cope with the more difficult term $R_5^\epsilon$. Recalling that $\zeta=\eta + l$, we have, by symmetry and $\widehat{\sigma_k}(m)\cdot m=0$
		\begin{align*}
			R_5^\epsilon &= \sum_{m,l}\widehat{\sigma_k}(m)\cdot \eta\; \widehat{\sigma_k}(l)\cdot\eta\;(\Phi_\epsilon-\Psi_\epsilon) P_{\eta+l}^\perp\hat{f}(\eta) \\
			&=\frac{1}{2}\sum_{m,l}\widehat{\sigma_k}(m)\cdot (\eta+m)\; \widehat{\sigma_k}(l)\cdot\eta\;(\Phi_\epsilon-\Psi_\epsilon) \big(P_{\eta+l}^\perp-P_{\eta+m}^\perp\big)\hat{f}(\eta)
		\end{align*}
		By \cite[Lemma 5.5]{FlaLuo21}, We have
		\begin{align*}
			\big|(\eta+m)\,\big(P_{\eta+l}^\perp-P_{\eta+m}^\perp\big)\big| 
			&\leq  |\eta+m|\times4\frac{|m-l|}{|\eta+m|}\lesssim |m|+|l|.
		\end{align*}
		implying that the difference of the non-local term absorbs one derivative. Consequently, the estimate for $R_5^\epsilon$ via a completely parallel argument to the one used for $R_4^\epsilon$, which completes the proof.
	\end{proof}
	\section{Solution properties of the limit equations}
	
	In this section, we establish the well-posedness of the limit equations under various settings. The truncation of the nonlinear term is omitted since it does not change the well-posedness results. 
	Consider the Navier-Stokes equation with critical hyperviscosity on $\Tb^3$:
	\begin{equation}\label{critical-hypervis}
		\left\{
		\begin{aligned}
			& \partial_t u + \Pi (u\cdot\nabla u) =-\nu(-\Delta)^{\frac{5}{4}}u,\\
			& u(0)=u_0.
		\end{aligned}
		\right.
	\end{equation}
	It is well known that global regularity holds for \eqref{critical-hypervis}. For completeness, we give the sketch of proof.
	\begin{proposition}
		Fix $s\geq0$. For every $u_0\in H^s$, the equation  \eqref{critical-hypervis} admits a unique solution $u\in C([0,\infty);H^s)\cap L^2((0,\infty);H^{s+1})$. Moreover, the solution instantaneously gains regularity in time and space 
		\begin{equation*}
			u\in C_{t,x}^\infty \quad \text{ for every } t>0,\;x\in \Tb^3.
		\end{equation*}
	\end{proposition} 
	\begin{proof}
		The proof follows from \cite[Section 2]{KP02}. First we give some a priori estimates.
		\begin{itemize}
			\item $L^2$ estimate. Due to the cancellation $\<u\cdot\nabla u, u\>=0$ by $\divg u=0$, we have
			\begin{equation*}
				\frac{\d}{\d t}\|u\|_{L^2}^2 + 2\nu\|u\|_{H^\frac{5}{4}}^2=0,
			\end{equation*}
			which implies
			\begin{equation*}
				\sup_{t\geq 0}\|u(t)\|_{L^2}^2 + \nu\int_{0}^{\infty}\|u(t)\|_{H^\frac{5}{4}}^2\,\d t\leq 2\|u_0\|_{L^2}^2.
			\end{equation*}
			Moreover, by Poincar\'e's inequality, we have exponential decay for $L^2$ norm:
			\begin{equation*}
				\|u(t)\|_{L^2}\leq e^{-\nu t}\|u_0\|_{L^2},\quad \text{ for every }t\geq 0.
			\end{equation*}
			\item $H^s$ estimate. Let $\Lambda=(-\Delta)^{\frac{1}{2}}$. Still by cancellation, we have
			\begin{equation*}
				\frac{1}{2}\frac{\d}{\d t}\|u\|_{H^s}^2 + \nu\|u\|_{H^{s+\frac{5}{4}}}^2= -\<\Lambda^s(u\cdot \nabla u),\Lambda^s u\>= -\<[\Lambda^s,u\cdot \nabla] u,\Lambda^s u\>.
			\end{equation*}
			By the classical Kato-Ponce estimate \cite{KatPon88} and Sobolev embedding $H^{\frac{5}{4}}\hookrightarrow L^{12},\;H^{\frac{1}{4}}\hookrightarrow L^{\frac{12}{5}}$, we have
			\begin{equation*}
				\|[\Lambda^s,u\cdot \nabla] u\|_{L^2}\lesssim_s \|u\|_{W^{s,12}}\|\nabla u\|_{L^\frac{12}{5}}\lesssim_s \|u\|_{H^{s+\frac{5}{4}}}\|u\|_{H^{\frac{5}{4}}}.
			\end{equation*}
			Hence we arrive at
			\begin{align*}
				\frac{1}{2}\frac{\d}{\d t}\|u\|_{H^s}^2 + \nu\|u\|_{H^{s+\frac{5}{4}}}^2&\leq 	\|[\Lambda^s,u\cdot \nabla] u\|_{L^2} \|\Lambda^s u\|_{L^2} \\
				&\leq C_s \|u\|_{H^{s+\frac{5}{4}}}\|u\|_{H^{\frac{5}{4}}}\|u\|_{H^s} \\
				&\leq \frac{\nu}{2} \|u\|_{H^{s+\frac{5}{4}}}^2 + \frac{C_s}{\nu} \|u\|_{H^{\frac{5}{4}}}^2\|u\|_{H^s}^2,
			\end{align*}
			yielding 
			\begin{align}
				&\sup_{t\geq 0}\|u(t)\|_{H^s}\leq \|u_0\|_{H^s}e^{\frac{C_s}{\nu}\int_{0}^{\infty}\|u(t)\|_{H^{5/4}}^2\,\d t}\leq \|u_0\|_{H^s} e^{\frac{C_s}{\nu^2}\|u_0\|_{L^2}^2},\label{s}\\
				&\int_{0}^{\infty} \|u(t)\|_{H^{s+\frac{5}{4}}}^2\,\d t\leq \|u_0\|_{H^s}^2 +\frac{C_s}{\nu^2} 	\sup_{t\geq 0}\|u(t)\|_{H^s}^2	 \|u_0\|_{L^2}^2.\label{s+5/4}
			\end{align}
			With these a priori estimates, it is easy to prove the existence of solutions in $u\in C([0,\infty);H^s)\cap L^2((0,\infty);H^{s+1})$. For the uniqueness, suppose $u^1, u^2$ are two solutions in $C([0,\infty);L^2)\cap L^2((0,\infty);H^{\frac{5}{4}})$ with the same initial data. Now the difference $u=u^1-u^2$ satisfies
			\begin{equation*}
				\left\{
				\begin{aligned}
					& \partial_t u + \Pi (u^1\cdot\nabla u) + \Pi (u\cdot\nabla u^2)=-\nu(-\Delta)^{\frac{5}{4}}u,\\
					& u(0)=0.
				\end{aligned}
				\right.
			\end{equation*}
			Then we have
			\begin{equation*}
				\frac{1}{2}\frac{\d}{\d t}\|u\|_{L^2}^2 + \nu\|u\|_{H^\frac{5}{4}}^2= -\<u\cdot\nabla u^2,u\>\leq \|\nabla u^2\|_{H^{\frac{1}{4}}} \|u\otimes u\|_{H^{-\frac{1}{4}}}.
			\end{equation*}
			By \cite[Corollary 2.55]{BCD11} and interpolation, we have  
			\begin{equation*}
				\|u\otimes u\|_{H^{-\frac{1}{4}}}\lesssim \| u\|_{H^{\frac{5}{8}}}^2\lesssim \| u\|_{L^{2}}\| u\|_{H^{\frac{5}{4}}}.
			\end{equation*} 
			Thus by Cauchy inequality, we obtain
			\begin{align*}
				\frac{1}{2}\frac{\d}{\d t}\|u\|_{L^2}^2 + \nu\|u\|_{H^\frac{5}{4}}^2&\leq C\| u^2\|_{H^{\frac{5}{4}}}\| u\|_{L^{2}}\| u\|_{H^{\frac{5}{4}}}\\
                &\leq \nu\|u\|_{H^\frac{5}{4}}^2 + C\| u^2\|_{H^{\frac{5}{4}}}^2\| u\|_{L^{2}}^2,
			\end{align*}
			which implies $u=0$ by Gr\"onwall's inequality. The instantaneously regularity can be obtained by bootstrap.\qedhere
		\end{itemize}
	\end{proof}
	\begin{remark}
		We have shown the exponential decay of $L^2$ norm, boundedness of $H^s$ norm and instantaneous regularity, thus we actually have exponential decay in every Sobolev space. In particular, we give an explicit decay in $H^s$: fix $R_0>0$, suppose $u_0\in H^s$ and $\|u_0\|_{H^s}\leq R_0$, then there exists a constant $K=K(s,\nu,R_0)$, such that
		\begin{equation}\label{s-exp}
			\|u(t)\|_{H^s}\leq K(s,\nu,R_0)e^{-\frac{5\nu}{4s+5}t},\quad \text{ for every }t\geq 0.
		\end{equation}
		Indeed, there must exist some $t_0\in (0,1)$, such that
		\begin{equation*}
			\|u(t_0)\|_{H^{\frac{5}{4}}}^2\leq\int_{0}^{\infty} \|u(t)\|_{H^{s+\frac{5}{4}}}^2\,\d t=:R_1,
		\end{equation*} 
		where $R_1$ is bounded by some constant $C_1(s,\nu,R_0)$ by \eqref{s+5/4}. Now by \eqref{s}, we know that for all $t>t_0$, $\|u(t)\|_{H^{s+\frac{5}{4}}}$ is bounded by  constant $C_2(s,\nu, R_1)$.
		By interpolation between $H^{s+\frac{5}{4}}$ and $L^2$, we obtain, for every $t\geq t_0$,
		\begin{equation*}
			\|u(t)\|_{H^s}\leq \|u(t)\|_ {H^{s+\frac{5}{4}}}^{\theta}\|u(t)\|_{L^2}^{1-\theta}\leq C_2(s,\nu, C_1)^\theta \|u_0\|_{L^2}^{\theta} e^{-(1-\theta)\nu t},
		\end{equation*}
		where $\theta=\frac{4s}{4s+5}\in (0,1)$.
	\end{remark}
	
	Next we consider 3D NS equation with damping:
	\begin{equation}\label{NS-damping}
		\left\{
		\begin{aligned}
			& \partial_t u + \Pi (u\cdot\nabla u) =\Delta u-\nu u,\\
			& u(0)=u_0.
		\end{aligned}
		\right.
	\end{equation}
	\begin{proposition}\label{limit-NS}
		Fix $R_0>0,\;s>\frac{1}{2}$. There exists $\nu_0=\nu_0(R_0)$ such that for every $\nu>\nu_0$ and initial data $u_0$ satisfying 
		\begin{equation*}
			u_0\in H^s,\quad\text{and} \quad \|u_0\|_{H^s}\leq R_0,
		\end{equation*}
		the equation \eqref{Euler-damping} admits a unique solution $u\in L^2([0,\infty); H^{s+1})\cap C([0,\infty);H^s)$. Moreover, if $u_0\in B_{2,p}^{n+1-\frac{2}{p}}$ for some $n>-\frac{1}{2},\, 1\leq p< \infty$ with $n+1-\frac{2}{p} >s$, then  the solution $u$ also belongs to $L_{loc}^p([0,\infty); H^{n+1})\cap C([0,\infty);B_{2,p}^{n+1-\frac{2}{p}})$.
	\end{proposition}
	
	\begin{proof}
		As the space $H^{s}$ is subcritical if $s>1/2$, it is classical and we only focus on energy estimates. Indeed, we have
		\begin{equation*}
			\frac{1}{2}\frac{\d}{\d t}\|u\|_{H^s}^2  + \|u\|_{H^{s+1}}^2+\nu \|u\|_{H^s}^2 = -\big\<\Lambda^s u, [\Lambda^s,u\cdot\nabla]u \big\>.
		\end{equation*}
		Define $x(t)=\|u(t)\|_{H^s}^2$. Similarly to the proof of Proposition \ref{small-GWP}, we get
		\begin{equation*}
			\dot{x} +\nu x\leq Cx^{p(s)},
		\end{equation*}
		where $p(s)=\frac{2(2s+1)}{2s-1}$ for $s\in (1/2,3/2)$ and $p(s)=4$ for $s\geq 3/2$. It is easy to check that when $\nu>\nu_0=CR_0^{2(p(s)-1)}$, $x(t)$ will not blow up and has exponential decay when $t\rightarrow \infty$.
		
		Under the assumption $n+1-\frac{2}{p} >s>1/2$, the embedding $B_{2,p}^{n+1-\frac{2}{p}}\hookrightarrow H^s$ holds; thus, the initial condition is subcritical to \eqref{NS-damping}. Hence it is standard that the local well posedness holds in space 
		\begin{equation*}
			L^p([0,T_*]; H^{n+1})\cap C([0,T_*];B_{2,p}^{n+1-\frac{2}{p}}),\quad \text{ for some } T_*>0
		\end{equation*}
		On the other hand, following the same bootstrap argument as in the proof of Proposition \ref{GWP-NS-cut-off}, the solution $u$ instantaneously gains regularity
		\begin{equation*}
			u\in L^2([\delta,\infty); H^{s'+1})\cap C([\delta,\infty);H^{s'}),\quad \text{ for every }s'>s \text{ and }\delta>0.
		\end{equation*}
		By the uniqueness of the solution to \eqref{NS-damping}, we obtain the desired result $u\in L_{loc}^p([0,\infty); H^{n+1})\cap C([0,\infty);B_{2,p}^{n+1-\frac{2}{p}})$.
	\end{proof}

	Now we turn to the 3D Euler equation with damping:
	\begin{equation}\label{Euler-damping}
		\left\{
		\begin{aligned}
			& \partial_t u + \Pi (u\cdot\nabla u) =-\nu u,\\
			& u(0)=u_0.
		\end{aligned}
		\right.
	\end{equation}
	\begin{proposition}\label{limit-Euler}
		Fix $R_0>0,\;s>\frac{5}{2}$. There exists $\nu_0=\nu_0(R_0)$ such that for every $\nu>\nu_0$ and initial data $u_0$ satisfying 
		\begin{equation*}
			u_0\in H^s,\quad\text{and} \quad \|u_0\|_{H^s}\leq R_0,
		\end{equation*}
		the equation \eqref{Euler-damping} admits a unique solution $u\in C([0,\infty);H^s)$.
	\end{proposition}
	\begin{proof}
		In the following, we restrict our attention to the energy estimates. Let $\Lambda^s=(-\Delta)^{\frac{s}{2}}$. At least formally, we have
		\begin{equation*}
			\frac{1}{2}\frac{\d}{\d t}\|u\|_{H^s}^2  + \nu \|u\|_{H^s}^2 = -\big\<\Lambda^s u, \Lambda^s(u\cdot\nabla u) \big\> = -\big\<\Lambda^s u, [\Lambda^s,u\cdot\nabla]u \big\>.
		\end{equation*} 
		It is clear that $\big|\big\<\Lambda^s u, [\Lambda^s,u\cdot\nabla]u \big\>\big|\lesssim_s \|u\|_{H^s}^3$  (c.f. \cite[Section 2.1]{BV22}), hence we arrive at
		\begin{equation*}
			\frac{\d}{\d t}\|u\|_{H^s}  + \nu \|u\|_{H^s} \leq C_s \|u\|_{H^s}^2,
		\end{equation*} 
		which implies that, for every $t\geq0$,
		\begin{equation*}
			\|u(t)\|_{H^s}\leq \frac{\nu}{C_s-\big(C_s-\frac{\nu}{\|u_0\|_{H^s}}\big)e^{\nu t}}.
		\end{equation*}
		It is obvious that the solution will not blow up if $R_0< \frac{\nu}{C_s}$, hence we can choose $\nu_0=C_s R_0$.
	\end{proof}
	\section{Pathwise uniqueness of (\ref{scalar eqn})}\label{linear-unique}
	We always assume that the covariance $Q(x,y)=\sum_k \sigma_k(x)\otimes \sigma_k(y)$ satisfies
	\begin{itemize}
		\item \makebox[0.7\textwidth][l]{$Q(x,y)=Q(x-y)$, for every $x,y \in \mathbb{T}^d$;} (homogeneous)
		\item \makebox[0.7\textwidth][l]{$\hat{Q}(\xi)=g(\xi)P_\xi^\perp$, where $g \in l^1(\mathbb{Z}_0^d) \cap l^\infty(\mathbb{Z}_0^d)$ is radial.} (isotropic)
	\end{itemize}
	
	By linearity of equation \eqref{scalar eqn}, proving pathwise uniqueness reduces to showing that any weak solution $\omega \in L^\infty(\Omega \times [0,T]; H^{2b+a})$ with zero initial data is necessarily trivial. With out loss of generality, we take $b=0$. Writing it into It\^o's form and taking Fourier transform, we have, for every $\xi\in \Z_0^d$,
	\begin{equation*}
		\begin{split}
			\d \omega_{\xi} = \frac{i}{(2\pi)^{\frac{d}{2}}}\sum_{k,\eta} \widehat{\sigma_k}(\xi-\eta)\cdot\eta |\eta|^{2a}\omega_\eta\,\d W^k 
			-\frac{S_\xi}{2} \omega_{\xi}, 
		\end{split}
	\end{equation*} 
	where $S_\xi= (2\pi)^{-d/2}\sum_\eta\eta^\perp \hat{Q}(\xi-\eta)\eta |\eta|^{2a} |\xi|^{2a}$ and $\omega_\xi=\hat{\omega}_\xi$.
	By It\^o's formula, we get
	\begin{align*}
		\d |\omega_{\xi}|^2&= \omega_{\xi}\,\d \bar{\omega}_\xi + \bar{\omega}_\xi\,\d \omega_{\xi} + \d\<\omega_{\xi}, \bar{\omega}_\xi\> \\
		&=\d M -S_\xi |\omega_\xi|^2 + \dfrac{1}{(2\pi)^{\frac{d}{2}}}\sum_\eta \eta^\perp \hat{Q}(\xi-\eta)\eta\, |\eta|^{4a} |\omega_\eta|^2,
	\end{align*}
	where $M$ is the martingale part. Now set $x_\xi=|\xi|^{2a}\E\big[|\omega_{\xi}|^2\big]$, then it satisfies the following linear infinite system of coupled ODEs:
	\begin{equation*}
		\dot{x}_\xi + S_\xi x_\xi = \dfrac{|\xi|^{2a}}{(2\pi)^{\frac{d}{2}}}\sum_\eta \eta^\perp \hat{Q}(\xi-\eta)\eta\, |\eta|^{2a}\, x_\eta.
	\end{equation*}
	Fix $t\in [0,T]$. Let $A_\xi=\int_{0}^{t}x_\xi(s)\,\d s$, then integrating in time, we arrive at
	\begin{equation*}
		x_\xi(t)+ S_\xi A_\xi = \dfrac{|\xi|^{2a}}{(2\pi)^{\frac{d}{2}}}\sum_\eta \eta^\perp \hat{Q}(\xi-\eta)\eta\, |\eta|^{2a}\, A_\eta.
	\end{equation*}
	
	By the assumption $\omega \in L^\infty(\Omega \times [0,T]; H^{a})$, we have $\sum_\xi A_\xi<\infty$, which implies that $\lim\limits_{|\xi|\rightarrow \infty}A_\xi=0$; in particular, we suppose $\{A_\xi\}$ attains its maximum at $\xi_1\in \Z_0^d$. We have
	\begin{align*}
		S_{\xi_1} A_{\xi_1}\leq x_{\xi_1}(t) + S_{\xi_1} A_{\xi_1}&= \dfrac{|\xi_1|^{2a}}{(2\pi)^{\frac{d}{2}}}\sum_\eta \eta^\perp \hat{Q}(\xi_1-\eta)\eta\, |\eta|^{2a}\, A_\eta \\
		&\leq \max A_\eta \dfrac{|\xi_1|^{2a}}{(2\pi)^{\frac{d}{2}}}\sum_\eta \eta^\perp \hat{Q}(\xi_1-\eta)\eta\, |\eta|^{2a} = A_{\xi_1} S_{\xi_1},
	\end{align*}
	which implies that all inequalities are qualities and in particular, for every $\eta$ such that $\xi_1-\eta\in \text{supp}(\hat{Q})$ and $\hat{Q}(\xi_1-\eta)\eta\neq 0$, it must hold $A_\eta=A_{\xi_1}$. Now pick $\zeta\in \text{supp}(\hat{Q})$. There exists a rotation $R$ such that $\xi_2:=\xi_1-R\zeta$ satisfies $|\xi_2|>|\xi_1|$ and $\hat{Q}(\xi_1-\eta)\eta\neq 0$. Note that $\xi_1-\xi_2=R\zeta\in \text{supp}(\hat{Q})$ since $\hat{Q}$ is isotropic; hence we have $A_{\xi_2}=A_{\xi_1}$. By iterating this procedure, we can construct a sequence ${\xi_n}$ such that $\xi_n\rightarrow \infty$ and $A_{\xi_n}= A_{\xi_1}=\max A_\eta$, which, combined with $\lim\limits_{|\xi|\rightarrow \infty}A_\xi=0$, implies that $A_{\xi}=0$ for every $\xi$.
	\bibliographystyle{amsalpha}
	\bibliography{scalinglimit}
\end{document}